\documentclass[11pt]{amsart}

\usepackage{fullpage,graphicx,amsfonts,amssymb,amsmath,amsthm}
\usepackage[all]{xy}
\usepackage[left=.8in,top=.8in,bottom=.8in,right=.8in,letterpaper]{geometry} 
\usepackage{mathtools}
\usepackage{graphicx}
\graphicspath{ {images/} }
\usepackage{enumerate}
\usepackage{hyperref} 
\usepackage{setspace}
\usepackage{amssymb} 
\usepackage{ esint }

\allowdisplaybreaks
\theoremstyle{plain} 
\newtheorem{theorem}    {Theorem}

\newtheorem{lemma}      [theorem]{Lemma}
\newtheorem{corollary}  [theorem]{Corollary}
\newtheorem{proposition}[theorem]{Proposition}

\theoremstyle{definition}
\newtheorem{definition} [theorem]{Definition}

\newtheorem{assumption}   [theorem]{Assumption}

\theoremstyle{remark}
\newtheorem{remark}              {Remark}

\renewcommand{\Re}{\operatorname{Re}}

\usepackage{url}

\begin{document} 

\title{Monotonicity formulas in positively curved settings with
applications to two-phase free boundary problems}

\author{Johannes Hosle \\ 09/12/2026}

\address{Department of Mathematics, Massachusetts Institute of Technology}
\email{jhosle@mit.edu}
\maketitle 

\begin{abstract}
Inspired by the monotonicity formula of Alt, Caffarelli, and Friedman, we consider a natural variant of the ACF functional in positively curved 2-dimensional settings where we integrate over sublevel sets of Green's function rather than disks. We prove sharp almost-monotonicity formulas for our new functional in the case of convex planar domains (both with the pole on the boundary and in the interior) and in the setting of 2-dimensional complete manifolds with Euclidean volume growth and nonnegative Gaussian curvature. The tools include the Schwarz-Christoffel formula and Riesz decomposition using conformal coordinates. As a consequence, using quasiconformal estimates, we give a new Lipschitz bound in the manifold setting for minimizers of the two-phase free boundary problem of Alt, Caffarelli, and Friedman, with the constant depending on only the total integral curvature, in contrast to work of Teixeira and Zhang, which gives constants depending on pointwise bounds for the Riemann curvature tensor and its derivatives. Our methods also recover the Lipschitz bound up to a Neumann boundary of Gemmer, Moon, and Raynor in the planar convex domain case.
\end{abstract}

\section{Introduction}

Inspired by the monotonicity formula of Alt, Caffarelli, and Friedman \cite{ACF84}, we introduce and study a new functional $\Phi$, where we integrate over sublevel sets of Green's function rather than balls as in the case of the Alt-Caffarelli-Friedman functional $\Phi_{\mathbb{R}^n}$. Recall that for $u_1, u_2 \in H^1_{\text{loc}}(\mathbb{R}^n)$, the ACF functional $\Phi_{\mathbb{R}^n}$, centered at $x_0 \in \mathbb{R}^n$, is defined by \begin{align}\label{Phi}
    \Phi_{\mathbb{R}^n}(r) &= \frac{1}{r^4} \int_{B_r(x_0)} \frac{|\nabla u_1(x)|^2}{|x-x_0|^{n-2}} \,dx \int_{B_r(x_0)} \frac{|\nabla u_2(x)|^2}{|x-x_0|^{n-2}} \,dx,
\end{align} where $B_r(x_0)$ is open ball of radius $r$ centered at $x_0$. The monotonicity of $\Phi_{\mathbb{R}^n}$ (for nonnegative subharmonic $u_1, u_2$ satisfying $u_1 u_2 = 0$) was proven and used by Alt, Caffarelli, and Friedman \cite{ACF84} to prove interior Lipschitz regularity for minimizers for the two-phase free boundary problem. In this work, we are motivated by the desire to prove Lipschitz regularity up to the Neumann boundary of a convex body $\Omega$, without requiring any further regularity assumptions on $\Omega.$ Since $|x|^{2-n}$ is the Green's function on $\mathbb{R}^n$ for $n \ge 3$, we believe the functional $\Phi$ is natural and of independent interest. 

The first setting we consider is an unbounded open convex set $\Omega$ in $\mathbb{R}^2$ containing a sector. Let $G$ be a Neumann Green's function on $\Omega$ with pole $p \in \overline{\Omega}$. We use the convention that $G(x) \sim \log |x - p|$ as $x\to p$. Let $\Omega_r := \{x \in \Omega: G(x) < \log r\}$ for $r \in (0, \infty)$. Let $u_1, u_2 \in H_{\text{loc}}^1({\Omega})$, and define \begin{equation}\label{Phi2D}
    \Phi(r) := \frac{1}{|\Omega_r|^2}\int_{\Omega_r}|\nabla u_1|^2 \,dx \int_{\Omega_r} |\nabla u_2|^2 \,dx,
\end{equation} where $|\Omega_r|$ denotes the area of $\Omega_r$. Since we are in dimension $2$, the Green's function does not appear as a factor in the integrand, but in the level sets that define the domain of integration, consistent with~\eqref{Phi}, where the $\log |x|$ weight is absent in dimension $n = 2.$ Our first main result is an almost-monotonicity formula for $\Phi$ in this setting.

\begin{theorem}\label{planaralmostmonPoleBoundaryChapter1} Let $\Omega$ be an unbounded, open convex set in $\mathbb{R}^2$ containing a sector of opening angle $\pi \alpha > 0$. Let $u_1, u_2 \in H^1_{\text{loc}}({\Omega})$ be continuous, nonnegative, and subharmonic functions in $\Omega$, with vanishing product $u_1u_2$. Let $p \in \overline{\Omega}$ with $u_1(p) = u_2(p) = 0$, and assume the weak Neumann boundary conditions $\partial_{\nu} u_i = 0$ on $\partial \Omega$. Let $G$ be the conformal Neumann Green's function with pole at $p$. Then, the function $\Phi$ \eqref{Phi2D} satisfies\begin{align*}
    \int_{\tau}^{T} \frac{\Phi'(r)}{\Phi(r)} \,dr &\ge -C + 2\log \alpha,
\end{align*} uniformly in $0 < \tau < T < \infty$, for an absolute constant $C \in \mathbb{R}.$
\end{theorem}

See~\eqref{Greensfunctiondefpboundary} and~\eqref{Greensfunctiondefpinterior} for the precise definition of $G$. The proof of Theorem~\ref{planaralmostmonPoleBoundaryChapter1} is broken up into two parts, namely
Proposition~\ref{planaralmostmonPoleBoundary} and Proposition~\ref{planaralmostmonPoleInterior}, corresponding to when the pole $p$ is on the boundary or in the interior of $\Omega$ respectively. A key tool in the proof is the Schwarz-Christoffel formula. We also examine the sharpness of Theorem~\ref{planaralmostmonPoleBoundaryChapter1}. We show in Proposition~\ref{sectoralmmonnearpole} that pure monotonicity $\Phi'(r) \ge 0$ cannot be obtained from the Wirtinger inequality eigenvalue estimate. In fact, in Proposition~\ref{sharpnessofACFpoleboundary}, we show that, using the Wirtinger inequality eigenvalue estimate, the $2\log \alpha$ in Theorem~\ref{planaralmostmonPoleBoundaryChapter1} cannot be replaced by $(2-\varepsilon) \log \alpha$ for any $\varepsilon > 0.$

In the case when $p \in \partial \Omega$, doubling the convex body yields a 2-dimensional Alexandrov surface without boundary, whose Green's function with pole at $p$ agrees with the original Green's function when restricted to $\Omega$, by symmetry. Therefore, a natural variant of the convex domain setting is the setting of 2-dimensional complete manifolds $M$ that have nonnegative Gaussian curvature and Euclidean volume growth. The Euclidean volume growth condition corresponds to the condition that the convex bodies contain a sector. For $u_1, u_2 \in H^1_{\text{loc}}(M)$, the functional $\Phi$ is defined by \begin{align}\label{Phimflddef}
    \Phi(r) := \frac{1}{|\Omega_r|^2} \int_{\Omega_r} |\nabla u_1|^2 \,d\operatorname{vol} \int_{\Omega_r} |\nabla u_2|^2 \,d\operatorname{vol}, 
\end{align} where $|\Omega_r|$ now denotes the Riemannian volume of the Green sublevel set $\{x \in M: G(x) < \log r\}$.

\begin{theorem}\label{mainthm2DmanifoldsChapter1} Let $M$ be a 2-dimensional, complete, connected, oriented, noncompact smooth Riemannian manifold without boundary, with nonnegative Gaussian curvature $K$, Euclidean volume growth, and finite total curvature. Let $\alpha = 1 - \frac{1}{2\pi}\int_{M} K \,d\operatorname{vol}$. Let $u_1, u_2 \in H^1_{\text{loc}}(M)$ be continuous, nonnegative, and subharmonic functions on $M$, with vanishing product. Let $p \in M$ with $u_1(p) = u_2(p) = 0$, and let $G$ be the conformal Green's function with pole at $p$. Then, the function $\Phi$~\eqref{Phimflddef} satisfies\begin{align*}
    \int_{\tau}^{T} \frac{\Phi'(r)}{\Phi(r)} \,dr &\ge -C + 2\log \alpha,
\end{align*} uniformly in $0 < \tau < T$, for an absolute constant $C \in \mathbb{R}.$
\end{theorem}

In~\eqref{alphadef}, we show that $\alpha > 0$. The conformal Green's function is defined via the conformal mapping between $M$ and $\mathbb{C}$, which we prove exists in Proposition~\ref{MdiffeoR2}. The role of the Schwarz-Christoffel formula in the proof is replaced by an analogous Riesz decomposition formula for subharmonic functions.  As with Theorem~\ref{planaralmostmonPoleBoundaryChapter1}, the constant $2$ in $2\log \alpha$ cannot be improved using the Wirtinger inequality eigenvalue estimate, as can be seen by doubling the sector example in Proposition~\ref{sharpnessofACFpoleboundary}. In Theorem~\ref{almostmonotFiala}, we offer another slightly weaker form of almost-monotonicity using the Fiala-Huber isoperimetric inequality \cite{Fiala}, \cite{HuberIso}. 

We then turn to applications of almost-monotonicity of $\Phi$ to the regularity theory of two-phase free boundary problems. 
\begin{theorem}\label{mfldregularitythmChapter1}
Let $M, \alpha$ be as in Theorem~\ref{mainthm2DmanifoldsChapter1}. Given $x_0 \in M$ and $u_0 \in H^1(B_2(x_0))$, let $u$ be a local minimizer of \begin{align*}
    J[v] &:= \int_{B_2(x_0)} \left(|\nabla v|^2 + 1_{\{v>0\}}\right)\,d\operatorname{vol}
\end{align*} over the set \begin{align*}
    v \in H^1(B_2(x_0)), \qquad \text{Tr }v = \text{Tr }u_0 \text{ on }\partial B_{2}(x_0).
\end{align*}Then there exists $C(\alpha) > 0$ so that \begin{align*}
    \sup_{B_{1/2}(x_0)}|\nabla u| &\le C(\alpha) \left(\int_{B_1(x_0)} |\nabla u|^2 \,d\operatorname{vol}\right)^{\frac{1}{2}}.
\end{align*}
\end{theorem}

The proof of this regularity theorem follows the general structure of Alt, Caffarelli, and Friedman \cite{ACF84}, although a new ingredient to accomodate our functional is the fact that sublevel sets of Green's function can be placed between two balls of comparable radii. Quasiconformal estimates, see Theorem~\ref{mfldquasisymmetry}, are used to obtain this. Our proof technique with the functional $\Phi$ can also be used to recover the Lipschitz bound of Gemmer, Moon, and Raynor \cite{GMR} for planar convex domains. 

We remark that ACF-type monotonicity formulas have previously been studied in the Riemannian setting. Teixeira and Zhang \cite{TeixeiraZhang} have extended the monotonicity formulas of Alt, Caffarelli, and Friedman \cite{ACF84} and Caffarelli, Jerison, and Kenig \cite{CJK} to Riemannian manifolds and proven the Lipschitz bound with implied constants depending on local geometry conditions, such as pointwise bounds on the Riemannian curvature tensor and its derivatives. On the other hand, our results give bounds solely in terms of the total integral curvature, namely the parameter $\alpha$.  

\section{Review of free boundary regularity theory}

We now review the fundamental contributions of Alt and Caffarelli \cite{AC81} and Alt, Caffarelli, and Friedman \cite{ACF84} to the regularity theory of one and two-phase free boundary problems respectively. The variational problem under consideration is the following. Let $\Omega \subset \mathbb{R}^n$ be an open, bounded, convex domain, $n\ge 2$, and $K\subset \overline{\Omega}$ be a closed set. For a given $u_0 \in H^1(\overline{\Omega})$, the function $u$ is defined to be the minimizer of the functional \begin{equation}\label{ACFJ}
    J[v] = \int_{\Omega}\left( |\nabla v|^2 + Q(x) 1_{\{v > 0\}}\right) \,dx
\end{equation} over the set $v \in H^1(\Omega)$ with $v = u_0$ on $K$, where $Q \in C^{\infty}(\overline{\Omega})$ is a smooth, positive function on $\overline{\Omega}$ and $1_{\{v > 0\}}$ is the indicator function of the set where $v$ is positive. Existence of minimizers follows from the direct method of the calculus of variations \cite{ACF84}. The Euler-Lagrange equation $J'[u] = 0$ gives rise to the following equations. If $\Omega_{+} = \{x\in \Omega: u > 0\}$ and $\Omega_{-} = \{x \in \Omega: u(x) \le 0\}^{\circ}$ denote the positive and negative phases respectively, then $u$ is harmonic in each phase, that is, $\Delta u = 0$ in $\Omega_{+}$ and in $\Omega_{-}$, interpreted in the weak sense. The one-phase case corresponds to when $u \ge 0$ on $\Omega$ and $u \equiv 0$ on $\Omega_{-}$. Otherwise, we say that the problem is two-phase. An additional equation satisfied by $u$ is the gradient jump condition \begin{equation}\label{jump}
    |\nabla u^{+}(x)|^2 - |\nabla u^{-}(x)|^2 = Q(x),
\end{equation} at all regular, interior points $x$ of the interface $F(u) := \partial \Omega^{+} \cap \partial \Omega^{-}$, known as the free boundary. In the one and two-phase cases respectively, Alt and Caffarelli \cite{AC81} and Alt, Caffarelli, and Friedman \cite{ACF84} proved that a minimizer $u$ is Lipschitz continuous in any compact subset of $\Omega\setminus K.$ The proof in the two-phase case uses in a key way the monotonicity of the ACF functional $\Phi_{\mathbb{R}^n}$ defined in~\eqref{Phi}. The
Lipschitz bound is crucial for the subsequent regularity theory of the free boundary, and, in particular, it is used to prove that the jump condition~\eqref{jump} is satisfied.

When Neumann boundary conditions are imposed, Gemmer, Moon, and Raynor \cite{GMR} proved Lipschitz regularity up to a Neumann boundary in dimension 2 using the ACF functional and its associated monotonicity properties. On the other hand, Gurevich \cite{Gur} proved that the Lipschitz bound up to a Dirichlet boundary frequently fails. 

An extension of the theorem of Gemmer, Moon, and Raynor was provided to higher dimensions by Beck, Jerison, and Raynor \cite{BJR}. They proved an almost-monotonicity formula for a variant of $\Phi_{\mathbb{R}^n}$ where they integrated over balls intersected with $\Omega$. The proof used a version of Friedland-Hayman's inequality \cite{FH}, the higher dimensional substitute for the Wirtinger inequality eigenvalue estimate, for convex cones. (See Beck and Jerison \cite{BJFH} for an alternative proof of this version of Friedland-Hayman, along with a characterization of the equality cases.) Using this almost-monotonicity formula, they proved the Lipschitz estimate up to a Neumann boundary for convex domains satisfying an appropriate Dini condition. In particular, the class of domains they considered included all $C^{1, \alpha}$-domains for any $\alpha > 0.$ However, they showed by example that there exist convex sets in dimension $n\ge 3,$ where, not only the Dini assumption fails, but one cannot conclude the desired Lipschitz bound by using their functional and the Friedland-Hayman estimate. This situation motivates our introduction of the functional $\Phi$. In the higher dimensional $n\ge 3$ manifold setting, if $M$ admits a positive Green's function $G$ and $\Omega_{r} = \{x \in M: G(x) > r^{2-n}\}$, then the natural functional $\Phi$ to study is given by \begin{align*}
    \Phi(r) = \frac{1}{f(r)^2} \int_{\Omega_r} |\nabla u_1(x)|^2 G(x) \,d\operatorname{vol} \int_{\Omega_r} |\nabla u_2(x)|^2 G(x) \,d\operatorname{vol},
\end{align*} where $f(r) := \int_{\Omega_r} G(x) \,d\operatorname{vol}$. In this work, we address only the case of dimension 2.

\section{Almost-monotonicity of $\Phi$ for planar convex domains}

\subsection{Green's function and the Schwarz-Christoffel formula}

Let $\Omega$ be an unbounded, open convex set in $\mathbb{R}^2$ containing a sector of opening angle $\pi \alpha > 0$. We may construct a Neumann Green's function on $\Omega$ by pulling back the explicit Neumann Green's functions on the upper half-plane under a conformal map. By the Riemann-mapping theorem, given $p \in \overline{\Omega}$, there exists a conformal map $\psi: \Omega \to \mathbb{H}$, extending continuously to a homeomorphism up to the boundary by Carath\'eodory's theorem. When $p \in \partial \Omega$, we normalize so that $\psi(p) = 0$, and when $p \in \Omega$, we normalize by $\psi(p) = i$. Let $\varphi = \psi^{-1}: \mathbb{H} \to \Omega$ denote the inverse map. It is routine to check that the functions \begin{align}\label{Greensfunctiondefpboundary}
G(x) &= \log |\psi(x)|, p \in \partial \Omega
\end{align} and
\begin{align}\label{Greensfunctiondefpinterior}
G(x) &= \log(|\psi(x)+i||\psi(x)-i|), p \in \Omega
\end{align} are Neumann Green's functions for $\Omega$. 

In the case when $\Omega$ is a polygon, the Schwarz-Christoffel formula states that \begin{align*}
    \varphi'(z) &= \prod_{j=1}^{N}(z-a_j)^{-\beta_i},
\end{align*} up to scaling by an absolute constant, where $a_j \in \mathbb{R}$ are the $\psi-$images of the vertices of $\Omega$, and $\pi \beta_i$ are the exterior angles at the vertex. Since $\Omega$ contains a sector of opening angle $\pi \alpha$, $\sum_{j=1}^{N} \beta_j \le 1 - \alpha.$ 

The Schwarz-Christoffel formula was generalized to our context of convex domains by Study \cite{Study1913} and Paatero \cite{Paatero1931}. In particular, these results imply that there exists a (positive) Borel measure $\mu$ on $\mathbb{R}$ with total variation $\mu(\mathbb{R}) \le 1 - \alpha$ such that\begin{align}\label{SchwarzChristoffelMeasure}
    \frac{\varphi'(z)}{\varphi'(i)} &= \exp\left(\int_{\mathbb{R}}\log\left(\frac{i - w}{z-w}\right)\,d\mu(w) \right).
\end{align}

\subsection{Almost-monotonicity for a pole on the boundary}

\begin{proposition}\label{planaralmostmonPoleBoundary} Let $\Omega$ be an unbounded, open convex set in $\mathbb{R}^2$ containing a sector of opening angle $\pi \alpha > 0$. Let $u_1, u_2 \in H^1_{\text{loc}}({\Omega})$ be continuous, nonnegative, and subharmonic functions in $\Omega$, with vanishing product. Let $p \in \partial \Omega$ with $u_1(p) = u_2(p) = 0$, and assume the weak Neumann boundary conditions $\partial_{\nu} u_i = 0$ on $\partial \Omega$. Let $G = \log |\psi(x)|$ be as in~\eqref{Greensfunctiondefpboundary}. Then, the function $\Phi$ \eqref{Phi2D} satisfies\begin{align*}
    \int_{\tau}^{T} \frac{\Phi'(r)}{\Phi(r)} \,dr &\ge -C + 2\log \alpha,
\end{align*} uniformly in $0 < \tau < T$, for an absolute constant $C \in \mathbb{R}.$
\begin{proof}
Here, and elsewhere, we give a formal proof, in the sense that the equality and inequalities with $\frac{\Phi'(r)}{\Phi'(r)}$ below in reality hold only almost everywhere. However, since $\log \Phi$ is absolutely continuous, the conclusion then follows by integration. Recall that \begin{align*}
    \Phi(r) &= \frac{1}{|\Omega_r|^2}\int_{\Omega_r}|\nabla u_1(x)|^2 \,dx \int_{\Omega_r}|\nabla u_2(x)|^2 \,dx.
\end{align*}By the change of variables formula, \begin{equation*}
    \int_{\Omega_r}|\nabla u_i|^2 \,dx = \int_{\psi(\Omega_r)}|\nabla u_i(\varphi(z))|^2|\varphi'(z)|^2 \,dz = \int_{B_r^{+}}|\nabla v_i(z)|^2 \,dz,
\end{equation*} where $v_i := u_i \circ \varphi$ and $B_r^{+}$ denotes $B_r \cap \mathbb{H}$, with $B_r := B_r(0)$. Then, \begin{align*}
    \frac{\Phi'(r)}{\Phi(r)} &= \frac{\frac{d}{dr}\int_{B_r^{+}}|\nabla v_1(z)|^2 \,dz}{\int_{B_r^{+}}|\nabla v_1(z)|^2 \,dz} + \frac{\frac{d}{dr}\int_{B_r^{+}}|\nabla v_2(z)|^2 \,dz}{\int_{B_r^{+}}|\nabla v_2(z)|^2 \,dz} - \frac{2|\Omega_r|'}{|\Omega_r|} \\ &= \frac{\int_{rS^1 \cap \mathbb{H}}|\nabla v_1(z)|^2 \,dz}{\int_{B_r^+}|\nabla v_1(z)|^2 \,dz} + \frac{\int_{rS^1 \cap \mathbb{H}}|\nabla v_1(z)|^2 \,dz}{\int_{B_r^+}|\nabla v_1(z)|^2 \,dz} - \frac{2|\Omega_r|'}{|\Omega_r|},
\end{align*} where $|\Omega_r|' = \frac{d}{dr}|\Omega_r|.$ Note that $v_i$ are also subharmonic on their supports and satisfy weak Neumann boundary conditions on $\partial \Omega$. Therefore, by Green's formula and a standard approximation argument, we obtain\begin{align*}
    \int_{B_r^{+}} |\nabla v_i(z)|^2 \,dz &\le \int_{rS^1 \cap \mathbb{H}} v_i \partial_{\nu} v_i \,d\sigma(z).
\end{align*}Define $\lambda_i(r)$ to be the largest constant such that \begin{align*}
    \int_{rS^1 \cap \mathbb{H}}|\nabla_T v(z)|^2 \,d\sigma(z) \ge \lambda_i(r)\int_{rS^1 \cap \mathbb{H}} v^2(z) \,d\sigma(z),
\end{align*} for all $v$ with the same support as $v_i$ on $rS^1 \cap \mathbb{H}$, with the interpretation that $v$ is free at the endpoints $\pm r$ if $v_i$ does not extend continuously there. If $\nabla_T, \nabla_N$ denote the tangential and normal gradients, \begin{align*}
    \int_{rS^1 \cap \mathbb{H}} |\nabla v_i|^2 \,d\sigma(z) &= \int_{rS^1 \cap \mathbb{H}} \left(|\nabla_T v_i(z)|^2 + |\nabla_N v_i(z)|^2\right) \,d\sigma(z) \\ &\ge \lambda_i(r) \int_{rS^1 \cap \mathbb{H}} v_i^2(z) \,d\sigma(z) + \int_{rS^1 \cap \mathbb{H}} |\nabla_N v_i|^2 \,d\sigma(z).
\end{align*}

For the next steps, through~\eqref{ACFreductionPhi}, we follow the argument in the proof of the ACF monotonicity formula \cite{ACF84}. Note that $(\partial_{\nu} v_i)^2 = |\nabla_N v_i|^2$. For any $t = t_i(r) > 0,$ we have \begin{align*}\begin{split}
    \int_{rS^1 \cap \mathbb{H}}v_i \partial_{\nu} v_i \,d\sigma(z) &= \frac{t}{2}\int_{rS^1 \cap \mathbb{H}}v_i^2 \,d\sigma(z) + \frac{1}{2t} \int_{rS^1 \cap \mathbb{H}}|\nabla_N v_i|^2 \,d\sigma(z) - \frac{1}{2}\int_{rS^1 \cap \mathbb{H}}\left(\sqrt{t}v_i - \frac{1}{\sqrt{t}}\partial_{\nu} v_i \right)^2 \,d\sigma(z) \\ &\le \frac{t}{2}\int_{rS^1 \cap \mathbb{H}}v_i^2 \,d\sigma(z) + \frac{1}{2t} \int_{rS^1 \cap \mathbb{H}}|\nabla_N v_i|^2 \,d\sigma(z).
\end{split}
\end{align*} Therefore, \begin{align*}
    \frac{\Phi'(r)}{\Phi(r)} &\ge \sum_{i=1}^{2} \left(\frac{\lambda_i(r) \int_{rS^1 \cap \mathbb{H}} v_i^2(z) \,d\sigma(z) + \int_{rS^1 \cap \mathbb{H}} |\nabla_N v_i|^2 \,d\sigma(z)}{\frac{t_i(r)}{2}\int_{rS^1 \cap \mathbb{H}}v_i^2 \,d\sigma(z) + \frac{1}{2t_i(r)} \int_{rS^1 \cap \mathbb{H}}|\nabla_N v_i|^2 \,d\sigma(z)} \right) - \frac{2|\Omega_r|'}{|\Omega_r|},
\end{align*} and choosing $t_i(r) = \sqrt{\lambda_i(r)}$, we obtain \begin{align}\label{FHreductionpolebdr}
    \frac{\Phi'(r)}{\Phi(r)} &\ge 2(\sqrt{\lambda_1(r)} + \sqrt{\lambda_2(r)}) - \frac{2|\Omega_r|'}{|\Omega_r|}.
\end{align} By an elementary Wirtinger inequality eigenvalue estimate, \begin{align*}
    \sqrt{\lambda_1(r)} + \sqrt{\lambda_2(r)} &\ge \frac{2}{r}.
\end{align*} Hence, for $r > 0$, \begin{align}\label{ACFreductionPhi}
    \frac{\Phi'(r)}{\Phi(r)} &\ge \frac{4}{r} - \frac{2|\Omega_r|'}{|\Omega_r|}.
\end{align} Moreover, by the change of variables formula and polar coordinates, \begin{equation}\label{OmegarChangeofVarFormula}\begin{split}
    |\Omega_r| = \int_{B_r^+} |\varphi'(z)|^2 \,dz = \int_{0}^{r} t h(t) \,dt, \qquad h(t) &:= \int_{0}^{\pi} |\varphi'(te^{i\theta})|^2 d\theta.
\end{split} 
\end{equation}

\textbf{Claim.} We have\begin{align}\label{convexplanaralpharecipbound}
    \frac{|\Omega_T|}{T^2} &\lesssim \frac{1}{\alpha} \frac{|\Omega_{\tau}|}{\tau^2}.
\end{align}Here, $f \lesssim g$ means $f \le Cg$ for an absolute positive constant $C$.

To see that this is sufficient to prove the theorem, recall that we must prove the inequality \begin{align}\label{desiredconclusionalmostmonbdr}
    \int_{\tau}^{T} \frac{\Phi'(r)}{\Phi(r)} \,dr &\ge -C + 2\log \alpha
\end{align} for any $0 < \tau < T < \infty$. By~\eqref{ACFreductionPhi}, \begin{equation*}
    -\int_{\tau}^{T} \frac{\Phi'(r)}{2\Phi(r)} \,dr \le \int_{\tau}^{T} \frac{|\Omega_r|'}{|\Omega_r|} - \frac{2}{r} \,dr = \log\left(\frac{|\Omega_T| \tau^2}{T^2|\Omega_{\tau}|} \right).
\end{equation*} Therefore,~\eqref{desiredconclusionalmostmonbdr} follows from~\eqref{convexplanaralpharecipbound}. It remains to prove the claim. We first treat measures whose support is both compact and
separated from the origin, and only afterward pass to a general measure by approximation.

\textit{Step 1: compact support separated from origin. }Fix $\varepsilon \in (0, 1)$ and assume that \begin{equation}\label{supportassump}
    \text{supp}(\mu) \subset A_{\varepsilon} = \{w \in \mathbb{R}: \varepsilon \le |w| \le \varepsilon^{-1}\}.
\end{equation} Since $\mu$ has support separated from the origin, the Schwarz-Christoffel formula~\eqref{SchwarzChristoffelMeasure} implies that $\varphi'$ extends holomorphically across a neighborhood of the origin and \begin{equation}\label{varphiprime0nonzero}
    \frac{\varphi'(0)}{\varphi'(i)} = \exp\left(\int_{\mathbb{R}}\log\left(\frac{i-w}{-w} \right) d\mu(w) \right) \neq 0.
\end{equation} In particular, \begin{equation*}
    \lim_{\tau \to 0^+} \frac{|\Omega_{\tau}|}{\tau^2} = \frac{\pi}{2} |\varphi'(0)|^2 \neq 0.
\end{equation*} We first prove the claim~\eqref{convexplanaralpharecipbound} with $\tau = 0^+$: \begin{align}\label{OmegaTbound}
    |\Omega_T| &\lesssim \frac{T^2}{\alpha}|\varphi'(0)|^2.
\end{align} In view of~\eqref{OmegarChangeofVarFormula}, it suffices to prove \begin{align*}
    \int_{0}^{T} h(r) dr &\lesssim \frac{T}{\alpha}|\varphi'(0)|^2.
\end{align*} Summing over dyadic intervals, it suffices to prove the estimate \begin{equation}\label{hintestimatetau0}
    \int_{I} h(r) dr \lesssim \frac{T}{\alpha} |\varphi'(0)|^2, \qquad I = [T, 2T].
\end{equation} By~\eqref{SchwarzChristoffelMeasure} and~\eqref{varphiprime0nonzero}, \begin{equation*}
    \frac{h(r)}{|\varphi'(0)|^2} = \int_{0}^{\pi} \exp\left(2\int_{\mathbb{R}}\log\left|\frac{w}{re^{i\theta} - w}\right| d\mu(w) \right) d\theta.
\end{equation*} For $|w| \notin [T/2, 4T]$, \begin{align*}
    \left|\frac{w}{re^{i\theta}-w}\right| \lesssim 1.
\end{align*} Let $J = [T/2, 4T]$, $\mu_J$ denote the restriction of $\mu$ to $J$, and $\mu_{J^{*}}$ denote the restriction to $J \cup (-J)$. Since $\mu$ is a finite measure, \begin{align*}
     \frac{h(r)}{|\varphi'(0)|^2} &\lesssim \int_{0}^{\pi} \exp\left(2\int_{\mathbb{R}} \log \left|\frac{w}{re^{i\theta}-w}\right| \,d\mu_{J^*}(w) \right) \,d\theta,
\end{align*} and \begin{align*}
    \frac{1}{|\varphi'(0)|^2}\int_{I} h(r) \,dr &\lesssim \int_{I} \int_{0}^{\pi}  \exp\left(2\int_{\mathbb{R}} \log \left|\frac{w}{re^{i\theta}-w}\right| \,d\mu_{J^*}(w) \right) \,d\theta \,dr.
\end{align*} Note that $\left|\frac{w}{re^{i\theta} - w}\right|^2 = \frac{w^2}{r^2 - 2wr\cos \theta + w^2}$. By the symmetry of cosine, it will suffice to estimate \begin{align*}
    \int_{I} \int_{0}^{\frac{\pi}{2}} \exp\left(\int_{\mathbb{R}} \log \frac{w^2}{r^2 - 2wr\cos \theta + w^2} \,d\mu_{J^*}(w) \right) \,d\theta \,dr.
\end{align*}On $\left[0, \frac{\pi}{2}\right]$, we have the elementary inequality $\cos \theta \le 1 - \frac{1}{4} \theta^2$. If $w \ge 0$, \begin{equation*}
    r^2 - 2wr\cos \theta + w^2 \ge r^2 - 2wr\left(1 - \frac{1}{4} \theta^2\right) + w^2 = (r-w)^2 + \frac{1}{2}wr\theta^2.
\end{equation*} If $w < 0$, we simply bound \begin{align*}
    r^2 - 2wr\cos \theta + w^2 \ge w^2. 
\end{align*} Therefore, since the integrand below is negative on $J^*\setminus J$, \begin{align*}
     \exp\left(\int_{\mathbb{R}} \log \frac{w^2}{r^2 - 2wr\cos \theta + w^2} \,d\mu_{J^*}(w) \right) &\le \exp\left(\int_{\mathbb{R}} \log \frac{w^2}{(r-w)^2 + \frac{1}{2}wr\theta^2} \,d\mu_{J}(w) \right) \\ &\le  \exp\left(\int_{J} \log K_w(r, \theta) \,d\mu(w) \right),
\end{align*} where $$K_{w}(r, \theta) := \frac{16T^2}{(r-w)^2 + \frac{1}{4}T^2\theta^2}.$$ 

Let $B := \mu(J) \le \mu(\mathbb{R}) \le 1 - \alpha.$ If $B = 0$, $\int_{0}^{\frac{\pi}{2}} \int_{I} \exp\left(\int_{J} \log K_{w}(r, \theta) \,d\mu(w) \right) \,dr \,d\theta \lesssim T$, so we assume $B > 0$. By Jensen's inequality, \begin{equation*}
    \exp\left(\int_{J} \log K_w(r, \theta) \,d\mu(w)\right) = \exp\left(\int_{J} \log K_w(r, \theta)^B \,\frac{d\mu(w)}{B}\right) \le \int_{J} K_w(r, \theta)^B \,\frac{d\mu(w)}{B}.
\end{equation*} Therefore, by Tonelli's theorem, \begin{align*}
    \int_{I} \int_{0}^{\frac{\pi}{2}}\exp\left(\int_J \log K_w(r, \theta) \,d\mu(w)\right) \,d\theta \,dr &\le \int_{J} \int_{I} \int_{0}^{\frac{\pi}{2}}  K_w(r, \theta)^B \,d\theta \,dr \,\frac{d\mu(w)}{B}.
\end{align*} Since $\frac{\mu}{B}$ restricts to a probability measure on $J$, it suffices to prove the estimate \begin{align}\label{ineqafterHolderplanar}
    \int_{I}\int_{0}^{\frac{\pi}{2}} K_w(r, \theta)^B \,d\theta \,dr &\lesssim \frac{T}{1 - B}
\end{align} uniformly in $w \in J.$ Set $u = \frac{r-w}{T}$. Then $dr = Tdu$, and $r \in [T, 2T], w \in [T/2, 4T]$ implies \begin{align*}
    u \in \left[\frac{T - w}{T}, \frac{2T - w}{T}\right] \subset \left[-3, \frac{3}{2}\right].
\end{align*} Hence, \begin{equation}\label{doubleintKProp4}
    \int_{I} \int_{0}^{\frac{\pi}{2}} K_w(r, \theta)^B \,dr \,d\theta \le T 16^B  \int_{-3}^{\frac{3}{2}} \int_{0}^{\frac{\pi}{2}}\left(u^2 + \frac{1}{4} \theta^2\right)^{-B} \,d\theta \,du \lesssim \frac{T}{1- B},
\end{equation} completing the proof of~\eqref{ineqafterHolderplanar},~\eqref{hintestimatetau0}, and~\eqref{OmegaTbound}.

We now prove~\eqref{convexplanaralpharecipbound} for arbitrary
$0<\tau<T<\infty$ under the support assumption~\eqref{supportassump}.
Decompose
\begin{equation*}
    \mu_{\tau/2} :=\mu\!\restriction_{\mathbb{R}\setminus(-\tau/2,\tau/2)}, \qquad
    \nu_{\tau/2}
    :=\mu\!\restriction_{(-\tau/2,\tau/2)},
\end{equation*}
and let $\varphi_{\tau/2}'$ be defined by the normalized representation
\eqref{SchwarzChristoffelMeasure}, with $\mu$ replaced by $\mu_{\tau/2}$ and
with the same denominator $\varphi_{\tau/2}'(i)=\varphi'(i)$. Write
\begin{equation*}
    h_{\tau/2}(r) :=\int_0^\pi|\varphi_{\tau/2}'(re^{i\theta})|^2\,d\theta,
    \qquad
    |\Omega_r^{(\tau/2)}|
    :=\int_0^r t h_{\tau/2}(t)\,dt.
\end{equation*}
Also set
\begin{equation*}
    m(\tau)
    :=\nu_{\tau/2}(\mathbb{R}),
    \qquad
    a_{\tau}
    :=\exp\left(
      2\int_{(-\tau/2,\tau/2)}\log|i-w|\,d\mu(w)
      \right).
\end{equation*}
The normalized formula gives
\begin{align*}
    |\varphi'(re^{i\theta})|^2
    &=M_{\tau/2}(r,\theta)
      |\varphi_{\tau/2}'(re^{i\theta})|^2,
\end{align*}
where
\begin{align*}
    M_{\tau/2}(r,\theta)
    &=a_{\tau}\exp\left(
      2\int_{(-\tau/2,\tau/2)}
      \log\frac{1}{|re^{i\theta}-w|}\,d\mu(w)
      \right).
\end{align*}
If $r\geq\tau$ and $|w|<\tau/2$, then
$|re^{i\theta}-w|\geq\tau/2$. Since $m(\tau)\leq1$, it follows that
\begin{align*}
    M_{\tau/2}(r,\theta)
    &\leq 4a_{\tau}\tau^{-2m(\tau)}.
\end{align*}
Consequently,
\begin{align}\label{OmegaTOmegatauPositiveBound}
    |\Omega_T|
    &\leq |\Omega_{\tau}|
      +4a_{\tau}\tau^{-2m(\tau)}|\Omega_T^{(\tau/2)}|.
\end{align}

In the opposite direction, if $0<r<\tau/2$ and $|w|<\tau/2$, then
$|re^{i\theta}-w|\leq\tau$, and hence
\begin{align*}
    M_{\tau/2}(r,\theta)
    &\geq a_{\tau}\tau^{-2m(\tau)}.
\end{align*}
The measure $\mu_{\tau/2}$ is compactly supported and has no mass in
$(-\tau/2,\tau/2)$. Thus $\varphi_{\tau/2}'$ extends holomorphically across
this interval. After multiplication by a unimodular constant, it is real on
the interval, so, by Schwarz reflection, it extends holomorphically to $B_{\tau/2}$. The mean value
inequality for the subharmonic function $|\varphi_{\tau/2}'|^2$ gives
\begin{align*}
    h_{\tau/2}(r)
    &\geq h_{\tau/2}(0)
      :=\pi|\varphi_{\tau/2}'(0)|^2,
      \qquad 0<r<\frac{\tau}{2}.
\end{align*}
It follows that
\begin{equation*}
    |\Omega_{\tau}|
    \geq a_{\tau}\tau^{-2m(\tau)}
      \int_0^{\tau/2}t h_{\tau/2}(t)\,dt
    \geq \frac18a_{\tau}\tau^{2-2m(\tau)}h_{\tau/2}(0),
\end{equation*}
or equivalently
\begin{align}\label{hOmegatauBound}
    a_{\tau}\tau^{-2m(\tau)}h_{\tau/2}(0)
    &\leq \frac{8|\Omega_{\tau}|}{\tau^2}.
\end{align}
Applying~\eqref{OmegaTbound} to $\mu_{\tau/2}$ gives
\begin{align}\label{OmegaThtau}
    |\Omega_T^{(\tau/2)}|
    &\leq \frac{C_1}{\alpha}T^2h_{\tau/2}(0)
\end{align}
for an absolute constant $C_1$. Combining
\eqref{OmegaTOmegatauPositiveBound}, \eqref{hOmegatauBound}, and
\eqref{OmegaThtau}, and using $\tau<T$ and $\alpha\leq1$, yields
\begin{equation*}
    \frac{|\Omega_T|}{T^2}
    \leq \frac{|\Omega_{\tau}|}{T^2}
      +\frac{32C_1}{\alpha}\frac{|\Omega_{\tau}|}{\tau^2} \leq \frac{C_0}{\alpha}\frac{|\Omega_{\tau}|}{\tau^2}.
\end{equation*}
This proves the claim under~\eqref{supportassump}.

\medskip
\noindent\emph{Step 2: approximation of a general measure.}
We first address the case of a possible atom at the origin. Write
$a=\mu(\{0\})$ and $\mu_0=\mu-a\delta_0$, and let $h_0$ and
$|\Omega_r^{(0)}|$ denote the quantities associated with $\mu_0$, using the
same normalization at $i$. The normalized representation gives
\begin{equation*}
    h(r)
    =r^{-2a}h_0(r),
    \qquad
    |\Omega_r|
    =\int_0^r t^{1-2a}h_0(t)\,dt.
\end{equation*}
Since $t\mapsto t^{-2a}$ is decreasing,
\begin{equation*}
    \frac{d}{dr}\log|\Omega_r|
    =\frac{r^{1-2a}h_0(r)}
      {\int_0^r t^{1-2a}h_0(t)\,dt}
    \leq \frac{rh_0(r)}{\int_0^r t h_0(t)\,dt}
     =\frac{d}{dr}\log|\Omega_r^{(0)}|.
\end{equation*}
After integration from $\tau$ to $T$, this shows that the ratio of the left and right hand sides appearing
in~\eqref{convexplanaralpharecipbound} for $\mu$ is no larger than the
corresponding quotient for $\mu_0$. It therefore suffices to consider the case
$\mu(\{0\})=0$.

For $\varepsilon \in (0, 1)$, define
\begin{align*}
    \mu_{[\varepsilon]}
    &:=\mu\!\restriction_{A_\varepsilon},    
\end{align*}
and set
\begin{align}\label{varphiPrimeR}
    \frac{\varphi_{[\varepsilon]}'(z)}{\varphi'(i)}
    &:=\exp\left(
      \int_{\mathbb{R}}
      \log\left(\frac{i-w}{z-w}\right)\,d\mu_{[\varepsilon]}(w)
      \right).
\end{align}
Thus $\varphi_{[\varepsilon]}'(i)=\varphi'(i)$. For every fixed
$z\in\mathbb{H}$, the function
\begin{align*}
    w\longmapsto\log\left|\frac{i-w}{z-w}\right|
\end{align*}
is bounded on $\mathbb{R}$ and tends to zero as $|w|\to\infty$.
Since $\mu(\{0\})=0$, dominated convergence therefore gives
$\varphi_{[\varepsilon]}'(z)\to\varphi'(z)$ pointwise in $\mathbb{H}$.

We record why this convergence also passes through the area integrals. Choose
$q>1$ so that $q(1-\alpha)<1$, and put
$b_\varepsilon=\mu_{[\varepsilon]}(\mathbb{R})$. If $b_\varepsilon>0$, Jensen's inequality gives
\begin{align*}
    \left|\frac{\varphi_{[\varepsilon]}'(z)}{\varphi'(i)}\right|^{2q}
    &\leq \int_{\mathbb{R}}
      \left|\frac{i-w}{z-w}\right|^{2qb_\varepsilon}
      \frac{d\mu_{[\varepsilon]}(w)}{b_\varepsilon}.
\end{align*}
For every fixed $r>0$,
\begin{align*}
    \sup_{w\in\mathbb{R}}
    \sup_{0\leq p\leq2q(1-\alpha)}
    \int_{B_r^+}\left|\frac{i-w}{z-w}\right|^p\,dz
    &<\infty,
\end{align*}
because $2q(1-\alpha)<2$: for bounded $w$ this is the elementary local
integrability of $|z-w|^{-p}$ in two dimensions, while for large $|w|$ the
ratio is uniformly bounded. Hence, $\{\varphi_{[\varepsilon]}'\}_\varepsilon$ is bounded in
$L^{2q}(B_r^+)$. The functions $|\varphi_{[\varepsilon]}'|^2$ are therefore uniformly
integrable on $B_r^+$, so pointwise convergence and Vitali's convergence theorem imply
\begin{align}\label{OmegaRConvergence}
    \lim_{R\to\infty}
    \int_{B_r^+}|\varphi_{[\varepsilon]}'(z)|^2\,dz
    &=\int_{B_r^+}|\varphi'(z)|^2\,dz
     =|\Omega_r|.
\end{align}
For each $\varepsilon$, Step 1 applied to $\mu_{[\varepsilon]}$ gives
\begin{align*}
    \frac{1}{T^2}\int_{B_T^+}|\varphi_{[\varepsilon]}'|^2\,dz
    &\leq \frac{C_0}{\alpha}\frac{1}{\tau^2}
      \int_{B_\tau^+}|\varphi_{[\varepsilon]}'|^2\,dz.
\end{align*}
Letting $\varepsilon\to 0$ and using~\eqref{OmegaRConvergence} proves
\eqref{convexplanaralpharecipbound} and completes the proof.

\end{proof}
\end{proposition}

\subsection{Level-set flow coordinates on the upper-half plane}\label{LevelSetFlowSection}
We now turn to the case when the pole $p$ is in the interior of $\Omega$. As a preliminary discussion, in this section, we analyze the model case of the upper half-plane with interior pole $p=i$. Writing
\[
G(z) = G_{\mathbb{H}}(z, i) = \log |F(z)|, \qquad F(z) = z^2 + 1,
\]
for the Neumann Green's function with pole at $p = i$, and $R=e^G$, we introduce coordinates adapted to the level-set flow of $R$. To be precise, given $z \in \mathbb{H}$, we define its $\theta$ coordinate to be the value in $[0, 2\pi)$ such that the flow \begin{align}\label{levelsetflowH}\begin{split}
    &\dot{z}(t) = \frac{\nabla R}{|\nabla R|^2}(z(t)), t > 0\\
    &z(0) = i,\\
    &\lim_{t\to 0^+} \frac{\dot{z}(t)}{|\dot{z}(t)|} = (\cos \theta, \sin \theta)
\end{split}
\end{align} passes through $z.$ Note that in~\eqref{levelsetflowH}, the gradient $\nabla G(z) = \frac{2\bar{z}}{\bar{z}^2 + 1}$ is non-vanishing in $\mathbb{H}\setminus \{i\}$. Therefore, $\nabla R = R \nabla G$ is also nonvanishing in $\mathbb{H}\setminus \{i\}.$ We can now verify that the flow~\eqref{levelsetflowH} is well defined, by giving an explicit formula for $\theta$. On a simply connected subset of $\mathbb{H}\setminus \{i\}$, we can choose a branch of the logarithm so that \begin{align*}
    \log F(z) &= \log |F(z)| + i \text{arg} F(z).
\end{align*} Thus, $|F(z)|, \text{arg} F(z)$ are harmonic conjugates. If we define $\Theta := \text{arg} F(z)$, which is well-defined up to a multiple of $2\pi$, then \begin{align}\label{ThetaHodgeStar}
    \,d\Theta = \star dG,
\end{align} where $\star$ is the Hodge star operator. Concretely, if $(u, v)$ are Euclidean coordinates on $\mathbb{H}$, $z = u + iv$, then \begin{equation*}
    dG = G_u du + G_v dv, \qquad d\Theta = -G_v du + G_u dv.
\end{equation*} Note that \begin{align*}
    \frac{d}{dt}R(z(t)) &= \nabla R(z(t)) \cdot \dot{z}(t) = 1,
\end{align*} by~\eqref{levelsetflowH}. Since $R(z(0)) = 0$, it follows that $R(z(t)) = t$ for all $t\ge 0$. On the other hand, \begin{align*}
    \frac{d}{dt}\Theta(z(t)) &= \nabla \Theta(z(t)) \cdot \dot{z}(t) = 0,
\end{align*} by~\eqref{ThetaHodgeStar} and~\eqref{levelsetflowH}. Therefore, $\Theta$ is constant along the flow. We have \begin{equation*}
    \Theta(z(t)) = \text{arg}(z^2(t) + 1) = \text{arg}(z(t)+i) + \arg(z(t)-i).
\end{equation*} Sending $t\to 0^+$, we obtain, from~\eqref{levelsetflowH}, \begin{equation*}
    \Theta(z(t)) = \text{arg}(2i) + \theta = \theta + \frac{\pi}{2}.
\end{equation*} Therefore, $\theta$ agrees with the harmonic conjugate of $G$, up to an additive constant. Since $G, \Theta$ are harmonic conjugates, they give conformal coordinates on $\mathbb{H} \setminus \{i\}$, with the metric given by \begin{align*}
    g &= |\nabla G|^{-2}(dG^2 + \,d\Theta^2).
\end{align*} Since $\,d\Theta = \,d\theta$, we obtain \begin{align}\label{zthetapartialIdentity}
    |\nabla G| |\partial_{\theta} z| &= 1.
\end{align}

\begin{lemma}\label{flowlines3pi/2} The only time a flow line $z(t)$ in~\eqref{levelsetflowH} hits $\partial\mathbb{H}$ at a finite time is when $\theta = \frac{3\pi}{2}$ and $t = 1$.
\begin{proof}
Assume first that $\theta = \frac{3\pi}{2}$. Since $\Theta$ is constant on the flow line, and $\Theta = \theta + \frac{\pi}{2} = 0$ (mod $2\pi$), we have $\text{arg}(z^2(t) + 1) = 0$ for all $z(t) \in \mathbb{H}$ on the flow line. This condition is equivalent to $1 + z^2(t) \in (0, \infty)$. If we write $z(t) = u(t) + iv(t)$, then we must have $u(t) = 0$ and $v(t) \in (0, 1)$. If we now use that $t = R(z(t)) = R(iv(t))$, we conclude that $z(t) = i \sqrt{1-t}$ for $t \in [0, 1]$. In particular, $z(1) = 0 \in \partial \Omega$.

Assume now that $\theta \neq \frac{3\pi}{2}.$ If $z(t) \in \partial \mathbb{H}$ for some finite $t = t_0$, then $\theta + \frac{\pi}{2} = \Theta(z(t_0)) = \text{arg}(z^2(t_0) + 1) = 0.$ This implies that $\theta = \frac{3\pi}{2}$, a contradiction.
\end{proof}
\end{lemma}

\subsection{Almost-monotonicity for a pole in the interior}

\begin{proposition}\label{planaralmostmonPoleInterior} Let $\Omega$ be an unbounded, open convex set in $\mathbb{R}^2$ containing a sector of opening angle $\pi \alpha > 0$. Let $u_1, u_2 \in H^1_{\text{loc}}({\Omega})$ be continuous, nonnegative, and subharmonic functions in $\Omega$ with vanishing product. Assume additionally the weak Neumann boundary conditions $\partial_{\nu} u_i = 0$ on $\partial \Omega$. Let $p \in \Omega$ with $u_1(p) = u_2(p) = 0$, and let $G = \log(|\psi(x)+i||\psi(x)-i|)$ be as in~\eqref{Greensfunctiondefpinterior}. Then, the function $\Phi$ \eqref{Phi2D} satisfies\begin{align*}
    \int_{\tau}^{T} \frac{\Phi'(r)}{\Phi(r)} \,dr &\ge -C + 2\log \alpha,
\end{align*} uniformly in $0 < \tau < T < \infty$, for an absolute constant $C \in \mathbb{R}.$
\begin{proof}
We have \begin{align*}
    \Phi(r) &= \frac{1}{|\Omega_r|^2}\int_{\Omega_r}|\nabla u_1(x)|^2 \,dx \int_{\Omega_r}|\nabla u_2(x)|^2 \,dx,
\end{align*}where, changing variables, we have \begin{equation*}
    \int_{\Omega_r}|\nabla u_i|^2 \,dx = \int_{\psi(\Omega_r)}|\nabla u_i(\varphi(z))|^2|\varphi'(z)|^2 \,dz = \int_{W_r}|\nabla v_i(z)|^2 \,dz,
\end{equation*} where $v_i = u_i \circ \varphi$ and $W_r$ denotes $\{z \in \mathbb{H}: G_{\mathbb{H}}(z, i) = \log(|z+i||z-i|) < \log r\}$. Let $R_{\mathbb{H}}(z, i) = \exp(G_{\mathbb{H}}(z, i)).$ By the coarea formula, \begin{align*}
    \frac{\Phi'(r)}{\Phi(r)} &= \frac{\frac{d}{dr} \int_{W_r} |\nabla v_1(z)|^2 \,dz}{\int_{W_r} |\nabla v_1(z)|^2 \,dz} + \frac{\frac{d}{dr} \int_{W_r} |\nabla v_2(z)|^2 \,dz}{\int_{W_r} |\nabla v_2(z)|^2 \,dz} - \frac{2|\Omega_r|'}{|\Omega_r|} \\ &= \frac{\int_{\partial W_r \cap \mathbb{H}} |\nabla v_1(z)|^2 \frac{\,d\sigma(z)}{|\nabla R_{\mathbb{H}}(z, i)|} }{\int_{W_r}|\nabla v_1(z)|^2 \,dz} + \frac{\int_{\partial W_r \cap \mathbb{H}} |\nabla v_2(z)|^2 \frac{\,d\sigma(z)}{|\nabla R_{\mathbb{H}}(z, i)|} }{\int_{W_r}|\nabla v_2(z)|^2 \,dz} - \frac{2|\Omega_r|'}{|\Omega_r|}.
\end{align*} Define $\lambda_i(r)$ to be the largest value so that \begin{align}\label{EigenvalueHpoleati}
    \int_{\partial W_r \cap \mathbb{H}} |\nabla_T v(z)|^2 \frac{\,d\sigma(z)}{|\nabla R_{\mathbb{H}}(z, i)|} &\ge \lambda_i(r) \int_{\partial W_r \cap \mathbb{H}} v^2(z) |\nabla R_{\mathbb{H}}(z, i)| \,d\sigma(z),
\end{align} for all $v \in H^1(\partial W_r \cap \mathbb{H})$ with the same support conditions as $v_i$. For $r < 1$, the  $v$ defining the eigenvalues are taken to have the same supports as $v_i$ on $\partial W_{r} \cap \mathbb{H}$. For $r \ge 1$, the same is true of the $v$, with the interpretation that the values of $v$ at the two (or one, when $r = 1$) points in $\partial({\partial W_r \cap \mathbb{H}})$ are free if $v_i$ does not extend continuously to $0$ there. Then, proceeding as in the proof of~\eqref{FHreductionpolebdr} in Proposition~\ref{planaralmostmonPoleBoundary}, we obtain \begin{align}\label{ACFprocHwithipole}
    \frac{\Phi'(r)}{\Phi(r)} &\ge 2(\sqrt{\lambda_1(r)} + \sqrt{\lambda_2(r)}) - \frac{2|\Omega_r|'}{|\Omega_r|}.
\end{align} For the two integrals in~\eqref{EigenvalueHpoleati}, we have, by~\eqref{zthetapartialIdentity}, \begin{equation*}
    \int_{\partial W_r \cap \mathbb{H}} |\nabla_T v(z)|^2 \frac{\,d\sigma(z)}{|\nabla R_{\mathbb{H}}(z, i)|} = \int_{0}^{2\pi} \left|\left\langle dv, \frac{1}{|\partial_{\theta} z|} \,d\theta \right\rangle \right|^2 \frac{1}{|\nabla R_{\mathbb{H}}|} |\partial_{\theta} z| \,d\theta = \frac{1}{r}\int_{0}^{2\pi} v'(\theta)^2 \,d\theta
\end{equation*} and \begin{equation*}
    \int_{\partial W_r \cap \mathbb{H}} v^2(z) |\nabla R_{\mathbb{H}}(z, i)| \,d\sigma(z) = \int_{0}^{2\pi} v^2(\theta) |\nabla R_{\mathbb{H}}||\partial_{\theta} z| \,d\theta = r \int_{0}^{2\pi} v^2(\theta) \,d\theta.
\end{equation*} Note that Lemma~\ref{flowlines3pi/2} implies that we may in fact write these as integrals on the full $[0, 2\pi]$ range, up to a singleton, which is of measure zero. Therefore, we may rewrite~\eqref{EigenvalueHpoleati} so that $\lambda_i(r)$ is the largest constant so that \begin{align*}
    \int_{0}^{2\pi} v'(\theta)^2 \,d\theta &\ge r^2 \lambda_i(r) \int_{0}^{2\pi} v^2(\theta) \,d\theta
\end{align*}for all $v \in H^1([0, 2\pi])$ with the same support conditions as $v_i$. For $r < 1$, the optimal partition is when the support of $v_1$ is $[0, \pi]$ and the support of $v_2$ is $[\pi, 2\pi],$ and the corresponding $v$'s are chosen as $\sin \theta$ and $-\sin \theta$ respectively. Therefore, \begin{align}\label{rleq1eigbound}
    \sqrt{\lambda_1(r)} + \sqrt{\lambda_2(r)} &\ge \frac{2}{r}, \text{   }r < 1.
\end{align}

For $r \ge 1$, $\theta = \frac{3\pi}{2}$ is free by Lemma~\ref{flowlines3pi/2}. The optimal choice occurs for $v = \cos\left(\frac{\theta}{2} + \frac{\pi}{4}\right)$ on $\left[-\frac{\pi}{2}, \frac{\pi}{2}\right]$ and $v = \sin\left(\frac{\theta}{2} - \frac{\pi}{4}\right)$ on $\left[\frac{\pi}{2}, \frac{3\pi}{2}\right]$, which gives \begin{align}\label{rgtr1eigbound}
    \sqrt{\lambda_1(r)} + \sqrt{\lambda_2(r)} &\ge \frac{1}{r}, \text{   }r \ge 1.
\end{align}

\textbf{Case 1.} $0 < \tau < T < 1.$ For $r \in [\tau, T]$, we have, by~\eqref{ACFprocHwithipole} and~\eqref{rleq1eigbound}, \begin{align*}
    \frac{\Phi'(r)}{\Phi(r)} &\ge \frac{4}{r} - \frac{2|\Omega_r|'}{|\Omega_r|}.
\end{align*} Therefore, \begin{equation*}
    -\int_{\tau}^{T} \frac{\Phi'(r)}{2\Phi(r)} \,dr \le \int_{\tau}^{T} \left(\frac{|\Omega_r|'}{|\Omega_r|} - \frac{2}{r} \right)\,dr = \log\left(\frac{|\Omega_T| \tau^2}{T^2|\Omega_{\tau}|} \right).
\end{equation*}We must prove \begin{align}\label{Case1VolIneqH}
    \frac{|\Omega_T|}{T^2} &\lesssim \frac{1}{\alpha}\frac{|\Omega_{\tau}|}{\tau^2}.
\end{align} The estimate~\eqref{Case1VolIneqH} will follow from \begin{align}\label{OmegaTCase1H}
    |\Omega_T| &\lesssim \frac{|\varphi'(i)|^2 T^2}{\alpha}
\end{align} and \begin{align}\label{OmegatauCase1H}
    |\Omega_{\tau}| &\gtrsim |\varphi'(i)|^2\tau^2.
\end{align} Note that $\varphi'(i)$ is nonvanishing, since $i$ is an interior point of $\mathbb{H}$ and $\varphi$ is conformal. We begin with the first inequality. By the change of variables formula, \begin{align*}
    |\Omega_T| &= \int_{W_{T}}|\varphi'(z)|^2 \,dz.
\end{align*} For $z \in \mathbb{H}$, $|z+i| > 1$. Therefore, \begin{align*}
    W_{T} \subset B_T(i),
\end{align*} and we have \begin{align*}
    \frac{|\Omega_T|}{|\varphi'(i)|^2} &\le \int_{B_T(i)} \frac{|\varphi'(z)|^2}{|\varphi'(i)|^2} \,dz.
\end{align*} Arguing as in the proof of Proposition~\ref{planaralmostmonPoleBoundary}, we conclude the estimate~\eqref{OmegaTCase1H}. For $\tau < 1$, we have the elementary inclusion \begin{align}\label{tau3Wtauinclusion}
    B_{\tau/3}(i) \subset W_{\tau}.
\end{align} For $z \in B_{\tau/3}(i)$, we have \begin{equation*}
    |z-w| \le |z-i| + |w-i| \le \frac{1}{3} + |w-i|.
\end{equation*} Since $|w - i| \ge 1$ for any $w \in \mathbb{R}$, \begin{align*}
    |z - w| \le \frac{4}{3}|w-i|.
\end{align*} For $z \in B_{\tau/3}(i),$ we therefore have \begin{equation*}
    \frac{|\varphi'(z)|^2}{|\varphi'(i)|^2} \ge \exp\left(2 \int_{\mathbb{R}} \log \left| \frac{i - w}{z - w}\right| \,d\mu(w) \right) \ge \exp\left(2\mu(\mathbb{R}) \log \frac{3}{4}\right) \ge \left(\frac{3}{4}\right)^2,
\end{equation*} because $\mu(\mathbb{R}) \le 1$ and $\log \frac{3}{4} < 0.$ Hence, by~\eqref{tau3Wtauinclusion}, \begin{equation*}
    \frac{|\Omega_{\tau}|}{|\varphi'(i)|^2} \ge \int_{B_{\tau/3}(i)} \left(\frac{3}{4}\right)^2 \,dz \gtrsim \tau^2.
\end{equation*} This proves~\eqref{OmegatauCase1H}, and finishes the proof of Case 1.

\textbf{Case 2.} $0 < \tau < 1 < T < \infty.$ For $r \in [\tau, 1),$ we have by~\eqref{ACFprocHwithipole} and~\eqref{rleq1eigbound}, \begin{align*}
    \frac{\Phi'(r)}{\Phi(r)} &\ge \frac{4}{r} - \frac{2|\Omega_r|'}{|\Omega_r|}.
\end{align*} On the other hand, for $r \in [1, T]$,~\eqref{ACFprocHwithipole} and~\eqref{rgtr1eigbound} yield \begin{align*}
    \frac{\Phi'(r)}{\Phi(r)} &\ge \frac{2}{r} - \frac{2|\Omega_r|'}{|\Omega_r|}.
\end{align*} Therefore, \begin{equation*}
    -\int_{\tau}^{T} \frac{\Phi'(r)}{2\Phi(r)} \,dr \le \int_{\tau}^{1} \left(\frac{|\Omega_r|'}{|\Omega_r|} - \frac{2}{r} \right)\,dr + \int_{1}^{T}\left(\frac{|\Omega_r|'}{|\Omega_r|} - \frac{1}{r}\right) \,dr = \log\left(\frac{|\Omega_T|\tau^2}{T|\Omega_{\tau}|^2}\right).
\end{equation*} Thus it suffices to prove \begin{align*}
    \frac{|\Omega_T|}{T} &\lesssim \frac{1}{\alpha}\frac{|\Omega_{\tau}|^2}{\tau^2},
\end{align*} which in view of~\eqref{OmegatauCase1H}, is implied by \begin{align}\label{Case2HOmegaTbound}
    |\Omega_T| &\lesssim \frac{|\varphi'(i)|^2 T}{\alpha}.
\end{align} Since $T > 1$, we have \begin{align}\label{WTcontained2sqrtTball}
    W_T \subset B_{2\sqrt{T}}^+.
\end{align} We may now run the argument that was used to prove~\eqref{OmegaTCase1H} to conclude~\eqref{Case2HOmegaTbound}.

\textbf{Case 3.} $1\le \tau < T < \infty.$ The bound to prove is now \begin{align}\label{Case3Hbound}
    \frac{|\Omega_T|}{T} &\lesssim \frac{1}{\alpha} \frac{|\Omega_{\tau}|}{\tau}.
\end{align}

\textbf{Subcase 1.} $1\le \tau \le 2.$ We have, by~\eqref{OmegatauCase1H}, \begin{equation*}
    |\Omega_{\tau}| \ge |\Omega_{\frac{1}{2}}| \gtrsim |\varphi'(i)|^2.
\end{equation*} Therefore, the desired bound~\eqref{Case3Hbound} follows from~\eqref{Case2HOmegaTbound}.

\textbf{Subcase 2.} $\tau \ge 2.$ We have\begin{align*}
    B_{\sqrt{\tau}/2}^+ \subset W_{\tau}.
\end{align*}In view of~\eqref{WTcontained2sqrtTball} as well, the estimate~\eqref{Case3Hbound} follows from the bound~\eqref{convexplanaralpharecipbound} in Proposition~\ref{planaralmostmonPoleBoundary}. This completes the proof of Case 3 and the theorem.
\end{proof}
\end{proposition}

\subsection{Sharpness of almost-monotonicity}\label{CounterexampleConvexPlanarSection}

The ACF monotonicity formula \cite{ACF84} is the beautiful pure monotonicity statement $\Phi_{\mathbb{R}^n}'(r) \ge 0.$ One may ask if the same is true for our functional $\Phi$. In the ACF argument, the proof reduces to the Friedland-Hayman inequality \cite{FH}. We followed this same framework in the proofs of Proposition~\ref{planaralmostmonPoleBoundary} and Proposition~\ref{planaralmostmonPoleInterior}. Indeed, in~\eqref{FHreductionpolebdr},~\eqref{ACFprocHwithipole}, we have the inequality \begin{align*}
    \frac{\Phi'(r)}{2\Phi(r)} &\ge \sqrt{\lambda_1(r)} + \sqrt{\lambda_2(r)} - \frac{|\Omega_r|'}{|\Omega_r|}.
\end{align*} For the pole on the boundary case, and the pole in the interior case for $r < 1$, we used \begin{align*}
    \sqrt{\lambda_1(r)} + \sqrt{\lambda_2(r)} &\ge \frac{2}{r},
\end{align*}so that\begin{align*}
    \frac{\Phi'(r)}{\Phi(r)} &\ge \frac{4}{r} - \frac{2|\Omega_r|'}{|\Omega_r|}.
\end{align*} What we show here is that the right-hand side can be negative, and, moreover, Proposition~\ref{planaralmostmonPoleBoundary} is sharp as an estimate for the right-hand side. While this does not directly imply that $\Phi'$ can be negative, it means that new ideas and a refinement of the ACF procedure would be needed in order to prove pure monotonicity of $\Phi$.

\begin{proposition}\label{sectoralmmonnearpole}
If $\Omega = \{te^{i\theta}: 0 < \theta < \pi \alpha, t> 0\}$ for $\alpha \in (0, 1)$ and $p = a \in \mathbb{R}_{+}$, then \begin{align*}
    \frac{4}{r} - \frac{2|\Omega_r|'}{|\Omega_r|} &< 0
\end{align*} for sufficiently small $r > 0$.
\begin{proof}
A conformal map $\psi: \Omega \to \mathbb{H}$ with pole at $a \in \mathbb{R}_{+}$ mapping $p$ to $0$ is given by $\psi(z) = (\alpha z)^{\frac{1}{\alpha}} - (\alpha a)^{\frac{1}{\alpha}},$ so that $G = \log |\psi|$. The inverse map is given by $\varphi(z) = \frac{1}{\alpha}(z + (\alpha a)^{1/\alpha})^{\alpha}$, so that $\varphi'(z) = (z + (\alpha a)^{1/\alpha})^{-(1-\alpha)}$. Recalling~\eqref{OmegarChangeofVarFormula}, we write $|\Omega_r| = \int_{0}^{r} t h(t) \,dt,$ where $h(t) = \int_{0}^{\pi} |\varphi'(te^{i\theta})|^2 \,d\theta.$ Note that the inequality $\frac{4}{r} - \frac{2|\Omega_r|'}{|\Omega_r|} < 0$ is equivalent to \begin{align}\label{hincineq}
    2 \int_{0}^{r} t h(t) \,dt &< r^2 h(r).
\end{align} Note that~\eqref{hincineq} will follow if $h$ is strictly increasing on $[0, r]$. We now show that $h'(0) = 0, h''(0) > 0$, which implies that $h$ is indeed strictly increasing in some such neighborhood $[0, r]$ for sufficiently small $r$. For simplicity, we set $c := (\alpha a)^{\frac{1}{\alpha}} > 0$ and $\beta := 1 - \alpha.$ Then, \begin{equation*}
    h(t) = \int_{0}^{\pi} |te^{i\theta} + c|^{-2\beta} \,d\theta = \int_{0}^{\pi} (t^2 + 2tc \cos \theta + c^2)^{-\beta} \,d\theta \\= c^{-2\beta}
    \int_{0}^{\pi}
    \left(1+2\frac{t}{c}\cos\theta+\frac{t^2}{c^2}\right)^{-\beta}
    \,d\theta.
\end{equation*} Using the Taylor expansion
\begin{equation*}
    (1+x)^{-\beta}
    =
    1-\beta x+\frac{\beta(\beta+1)}{2}x^2+O(x^3), \qquad x=2\frac{t}{c}\cos\theta+\frac{t^2}{c^2},
\end{equation*}
we obtain
\begin{align*}
    h(t)
    &=
    \pi c^{-2\beta}
    +
    \pi \beta^2 c^{-2\beta-2} t^2
    +
    O(t^3).
\end{align*}
In particular,
\begin{equation*}
    h'(0)=0,
    \qquad
    h''(0)=2\pi\beta^2 c^{-2\beta-2}>0.
\end{equation*}
\end{proof}
\end{proposition}

In fact, the estimate for $\int_{\tau}^{T}\frac{4}{r} - \frac{2|\Omega_r|'}{|\Omega_r|} \,dr$ in Proposition~\ref{planaralmostmonPoleBoundary} is sharp, as we confirm in the next proposition.

\begin{proposition}\label{sharpnessofACFpoleboundary}
If $\Omega = \{t e^{i\theta}: 0 < t < \pi \alpha\}$ for $\alpha \in (0, 1)$, $p = a \in \mathbb{R}_{+},$ and $T = (\alpha a)^{\frac{1}{\alpha}},$ then, \begin{align*}
    \int_{0^+}^{T}\left(\frac{4}{r} - \frac{2|\Omega_r|'}{|\Omega_r|}\right) \,dr &\le -C + 2\log \alpha
\end{align*} for an absolute constant $C \in \mathbb{R}.$
\begin{proof}
For any $T > 0$, we have \begin{equation*}
   \int_{0^+}^{T} \left(\frac{|\Omega_r|'}{|\Omega_r|} - \frac{2}{r}\right) \,dr = \log\left(\frac{|\Omega_T|}{T^2}\right) - \lim_{r \to 0} \log\left(\frac{|\Omega_r|}{r^2}\right) = \log\left(\frac{|\Omega_T|}{\pi |\varphi'(0)|^2 T^2}\right),
\end{equation*} where $\varphi'(0) = (\alpha a)^{-\frac{1-\alpha}{\alpha}} \neq 0$ as in the proof of Proposition~\ref{sectoralmmonnearpole}. We use the notation of the previous proposition throughout. Therefore, our desired inequality is equivalent to \begin{align*}
    \frac{|\Omega_T|}{|\varphi'(0)|^2 T^2} &\gtrsim \frac{1}{\alpha}.
\end{align*} We write $|\Omega_T| = \int_{0}^{T} t h(t) \,dt.$ Clearly, it is enough to show that $\frac{\int_{T/2}^{T} t h(t) \,dt}{|\varphi'(0)|^2 T^2} \gtrsim \frac{1}{\alpha}$, or equivalently, \begin{align}\label{inthlowerboundsector}
    \frac{1}{h(0)}\int_{T/2}^{T} h(t) \,dt &\gtrsim \frac{T}{\alpha}.
\end{align} We choose $T = c$. Then, \begin{align*}
    \frac{1}{h(0)}\int_{c/2}^{c} h(t) \,dt &= \frac{c^{2\beta}}{\pi} \int_{c/2}^{c} \int_{0}^{\pi} (t^2 + 2tc \cos\theta + c^2)^{-\beta} \,d\theta \,dt.
\end{align*} Let us make the change of variables $t = cs$. Then, \begin{align*}
    \frac{1}{h(0)}\int_{c/2}^{c} h(t) \,dt &= \frac{c}{\pi} \int_{1/2}^{1}\int_{0}^{\pi} (s^2 + 2s \cos \theta + 1)^{-\beta} \,d\theta \,dt.
\end{align*} Using the Taylor expansion of cosine and arguing as in the proof of Proposition~\ref{planaralmostmonPoleBoundary}, we obtain \begin{align*}
    \frac{1}{h(0)}\int_{c/2}^{c} h(t) \,dt &\asymp \frac{c}{\alpha}
\end{align*} as desired.
\end{proof}
\end{proposition}

We end by providing an analogous result to Proposition~\ref{sectoralmmonnearpole} for the case where the pole is in the interior.

\begin{proposition}
Let $\Omega = \mathbb{H}$ and $p = i$. Then, \begin{align*}
    \frac{4}{r} - \frac{2|\Omega_r|'}{|\Omega_r|} &< 0
\end{align*} for sufficiently small $r > 0.$
\begin{proof} Consider $r < 1.$ Let $\upsilon: B_{r} \to \Omega_r, w \to \sqrt{w - 1},$ which is well defined, because $\Re(1-w) > 0$ and $1-w$ can be made to avoid a branch cut of the logarithm. Then, \begin{equation*}
    |\Omega_r| = \int_{B_r} |\upsilon'(w)|^2 \,dw = \int_{B_r} \frac{1}{4|w-1|} \,dw.
\end{equation*} If we set $\xi(t) = \int_{0}^{2\pi}\frac{1}{4|te^{i\theta} - 1|} \,d\theta$, then \begin{align*}
    |\Omega_r| &= \int_{0}^{r} t \xi(t) \,dt.
\end{align*}From the computations in the proof of Proposition~\ref{sectoralmmonnearpole}, we have $\xi'(0) = 0, \xi''(0) > 0$, from which the desired inequality follows.
\end{proof}
\end{proposition}

\section{Almost-monotonicity of $\Phi$ for 2D manifolds}

\begin{assumption}\label{manifoldAssumption}
$M$ shall denote a 2-dimensional, complete, connected, oriented, noncompact smooth Riemannian manifold without boundary. Additionally, we assume that $M$ has nonnegative Gaussian curvature $K$, finite total curvature, and Euclidean volume growth.
\end{assumption}

We will use the notation $\alpha = 1 - \frac{1}{2\pi} \int_{M} K \,d\operatorname{vol}$, which we show in~\eqref{alphadef} is strictly positive. In the doubled convex domain model, this $\alpha$ corresponds to the largest $\alpha$ for which a sector of opening angle $\pi \alpha$ is contained in the convex domain.

\subsection{Conformal Green's function}\label{ConformalGreenSection}

\begin{proposition}\label{MdiffeoR2}
Any manifold $M$ satisfying Assumption~\ref{manifoldAssumption} is necessarily conformally equivalent to $\mathbb{C}$. 
\begin{proof}
Since $M$ is oriented and 2-dimensional, the metric determines a Riemann surface structure on $M$ with local isothermal coordinates. We have \begin{align*}
    K^{-} := \max\{-K, 0\} = 0,
\end{align*} since $K$ is nonnegative. Therefore, \begin{equation*}
    \int_{M} K^{-} \,d\operatorname{vol} = 0 < \infty.
\end{equation*} By work of Huber \cite{Hubersubharmonicdiffgeo}, it follows that $M$ is parabolic and of finite type, meaning \begin{align*}
    M \cong \widehat{M} \setminus \{p_1, .., p_k\}
\end{align*} for a compact $\widehat{M}$ and finitely many distinct points $p_i \in \widehat{M}, 1 \le i\le k.$ Note that we can assume we are deleting at least one point, since otherwise $M$ would be compact, contradicting Euclidean volume growth. Let $\gamma$ denote the genus of $\widehat{M}$ so that \begin{align*}
    \chi(M) &= 2 - 2\gamma - k.
\end{align*} By the Cohn-Vossen inequality \cite{CohnVossen1935}, \begin{equation*}
    \chi(M) \ge \frac{1}{2\pi}\int_{M} K \,d\operatorname{vol} \ge 0,
\end{equation*} where we use $K \ge 0$ in the last passage. Therefore, the only possibilities for $(\gamma, k)$ are $(\gamma, k) = (0, 1), (0, 2).$ 

In the first case, $M$ is diffeomorphic to the plane $\mathbb{R}^2$. In particular, $M$ is simply connected, and, by the uniformization theorem, is conformally equivalent to either the complex plane or the open disk. However, $M$ is parabolic, while the open disk is hyperbolic. Therefore, in this first case, $M$ is conformally equivalent to $\mathbb{C}.$

It remains to rule out the case $(\gamma, k) = (0, 2)$, where $M$ is diffeomorphic to a cylinder. In this case, $\chi(M) = 0$ and the Cohn-Vossen inequality implies that $K = 0$ identically. Hence, $M$ is a complete flat cylinder, contradicting Euclidean volume growth. This completes the proof.
\end{proof}
\end{proposition}

The fact, used in the proof of Propostion~\ref{MdiffeoR2}, that $M$ is of finite type allows us to conclude the positivity of $\alpha$ as well. Indeed, by a result of Shiohama \cite{Shiohama}, extending earlier work of Hartman \cite{Hartman} and Fiala \cite{Fiala}, any manifold $M$ satisfying Assumption~\ref{manifoldAssumption} satisfies\begin{align*}
    2\pi \chi(M) - \int_{M} K \,d\operatorname{vol} &= \lim_{t\to \infty} \frac{2\operatorname{vol}(B_t(p))}{t^2},
\end{align*} for any fixed $p \in M$, where $B_t(p)$ denotes the geodesic ball of radius $t$ centered at $p$. Note that $\lim_{t\to \infty} \frac{\operatorname{vol}(B_t(p))}{t^2}$ exists by the Bishop-Gromov theorem. Since $M$ has Euclidean volume growth, it follows that $\lim_{t\to \infty} \frac{\operatorname{vol}(B_t(p))}{t^2} > 0$. Furthermore, $\chi(M) = 1$ since $M$ is diffeomorphic to $\mathbb{R}^2$. Therefore, the quantity \begin{align}\label{alphadef}
    \alpha := 1- \frac{1}{2\pi} \int_{M} K \,d\operatorname{vol} > 0.
\end{align} Let $F: M \to \mathbb{C}$ denote an orientation-preserving conformal diffeomorphism with $F(p) = 0$, guaranteed to exist by Proposition~\ref{MdiffeoR2}. We define the \textit{conformal Green's function} with pole at $p$ by \begin{align*}
    G(x) := \log |F(x)|,
\end{align*} up to an additive constant, which we set to zero. That $G$ is indeed a Green's function can be proven as for the conformal Green's function in the planar case~\eqref{Greensfunctiondefpboundary},~\eqref{Greensfunctiondefpinterior}. We now construct conformal coordinates on $M$ that extend the coordinates we constructed in the upper half-plane in Section~\ref{LevelSetFlowSection}.

\begin{proposition}\label{GthetaCoordsProp} Let $(R, \theta)$ and $(u, v)$ be polar and Cartesian coordinates for $F$ respectively. Then: \begin{enumerate} \item \(G\) has no critical points on \(M\setminus\{p\}\). \item The map \begin{equation*} M\setminus\{p\}\to \mathbb R\times\mathbb T, \qquad x\mapsto (G(x),\theta(x)), \end{equation*} is a diffeomorphism. \item The angular one-form satisfies \begin{equation*} d\theta=\star dG. \end{equation*} \item The metric satisfies \begin{equation}\label{GthetaCoords} g=|\nabla G|^{-2}(dG^2+d\theta^2) \end{equation} on \(M\setminus\{p\}\). \item In the global isothermal coordinates \((u,v)\), the metric satisfies \begin{equation}\label{uvIsothermalCoords} g=|\nabla R|^{-2}(du^2+dv^2) \end{equation} on \(M\), with the conformal factor extending smoothly and positively across \(p\). \end{enumerate} \begin{proof} Since \(F\) is a diffeomorphism and \(F(p)=0\), it restricts to a diffeomorphism \begin{equation*} F:M\setminus\{p\}\to\mathbb C^*. \end{equation*} The polar-coordinate map \begin{equation*} \mathbb C^*\to \mathbb R\times\mathbb T, \qquad z\mapsto(\log|z|,\arg z), \end{equation*} is a diffeomorphism, with inverse \begin{equation*} (t,\vartheta)\mapsto e^{t+i\vartheta}. \end{equation*} Therefore \begin{equation*} x\mapsto (G(x),\theta(x)) \end{equation*} is a diffeomorphism from \(M\setminus\{p\}\) onto \(\mathbb R\times\mathbb T\). The function \(G\) has no critical points on \(M\setminus\{p\}\). Indeed, on \(\mathbb C^*\), $\log |z|$ has no critical points. Since \(G=\log|F|\) and \(F\) is a diffeomorphism, it follows that $G$ has no critical points on \(M\setminus\{p\}\). Next we identify the angular one-form. In the Euclidean plane, with its standard orientation, \begin{equation*} \star_{\mathbb C}d\log|z|=d\arg z \end{equation*} on \(\mathbb C^*\). Since \(F\) is orientation-preserving and conformal, the Hodge star on one-forms is preserved by pullback. Hence \begin{equation*} \star dG=d\theta \end{equation*} on \(M\setminus\{p\}\). It remains to compute the metric. The identity~\eqref{GthetaCoords} follows immediately from the second and third items. We now turn to~\eqref{uvIsothermalCoords}. Since \(F\) is conformal, \((u,v)\) are global isothermal coordinates, so \begin{equation*} g=\lambda^2(du^2+dv^2) \end{equation*} for some smooth positive function \(\lambda\) on \(M\). On \(M\setminus\{p\}\), $R=(u^2+v^2)^{1/2},$ and hence the Euclidean gradient of \(R\) has norm one. Since gradient norms scale by \(\lambda^{-1}\) under the conformal change $g=\lambda^2(du^2+dv^2),$ we obtain $|\nabla R|=\lambda^{-1},$ where the norm is taken with respect to the metric $g$. Consequently, \begin{equation*} g=|\nabla R|^{-2}(du^2+dv^2) \end{equation*} on \(M\setminus\{p\}\). Since \(\lambda\) is smooth and positive on all of \(M\), the conformal factor \begin{equation*} |\nabla R|^{-2}=\lambda^2 \end{equation*} extends smoothly and positively across \(p\).\end{proof} \end{proposition} 

When we have global conformal coordinates, there is a standard representation formula for the Gaussian curvature in terms of the conformal factor. By, for example, Theorem 13.1.3 in Dubrovin, Fomenko, and Novikov \cite{DFN}, and~\eqref{uvIsothermalCoords}, we obtain the following:

\begin{corollary}\label{CurvatureLaplacianIdentity} $\Delta \log |\nabla R| = K$.
\end{corollary}

\begin{remark}
An alternative approach to constructing conformal coordinates on $M$, that avoids the uniformization theorem, is to use the Li-Tam Green's function \cite{LiTamSymmetricGreensFcn} constructed by compact exhaustion. In fact, in our setting, the Li-Tam Green's function is unique and equals the conformal Green's function. By a result of Enciso and Peralta-Salas \cite{EncisoPeraltaSalas}, the Li-Tam Green's function on $M$ has no critical points, since $M$ is diffeomorphic to $\mathbb{R}^2$. Therefore, defining $d\theta = \star dG$, we obtain the conformal coordinates~\eqref{GthetaCoords}.
\end{remark}

\subsection{Riesz decomposition}

We now obtain a Riesz decomposition representation for
$\log |\nabla R|$ that will take the place of the Schwarz--Christoffel
formula in the almost-monotonicity proof in the manifold setting.

We identify $M$ with $\mathbb C$ using the conformal diffeomorphism $F$,
and write
\begin{equation*}
    h(z)
    :=
    \log|\nabla R|\bigl(F^{-1}(z)\bigr).
\end{equation*}
By~\eqref{uvIsothermalCoords},
\begin{equation*}
    g=e^{-2h}|dz|^2.
\end{equation*}
Moreover, Corollary~\ref{CurvatureLaplacianIdentity} implies that $h$ is a subharmonic function on all of
$\mathbb C$.

For a subharmonic function $v$ on $\mathbb C$, set
\begin{equation*}
    M(r,v)
    :=
    \max_{|z|=r}v(z),
    \qquad
    \beta(v)
    :=
    \lim_{r\to\infty}\frac{M(r,v)}{\log r}.
\end{equation*}
Belegradek and Hu~\cite{BelegradekHu} proved that the metric
$e^{-2v}|dz|^2$ is complete if and only if $\beta(v)\le 1$. Since
\begin{equation*}
    g=e^{-2h}|dz|^2
\end{equation*}
is complete, we have $\beta(h)\le 1$. Thus, for every
$\varepsilon>0$,
\begin{equation}\label{hlogupperbound}
    h(z)
    \le
    (1+\varepsilon)\log |z|
\end{equation}
for all sufficiently large $|z|$.

We next record an elementary estimate for the logarithmic kernel.

\begin{lemma}\label{LogKernelCircularL1Lemma}
There exists a universal constant $C \in \mathbb{R}$ such that, for every
$r\ge 1$ and every $\zeta\in\mathbb C$,
\begin{equation*}
    \frac{1}{2\pi}
    \int_0^{2\pi}
    \left|
        \log|re^{i\theta}-\zeta|-\log^+|\zeta|
    \right|
    \,d\theta
    \le
    C(1+\log r).
\end{equation*}
\end{lemma}

\begin{proof}
Set \(a=\zeta/r\) and
\[
    k_a(\theta)
    :=
    \log|e^{i\theta}-a|-\log^+|a|.
\]
By Jensen's formula,
\[
    \frac{1}{2\pi}\int_0^{2\pi}k_a(\theta)\,d\theta=0.
\]
Moreover, \(k_a\le \log 2\): this follows from
\[
    |e^{i\theta}-a|\le 2
\]
when \(|a|\le 1\), while for \(|a|>1\),
\[
    k_a(\theta)
    =
    \log\left|1-\frac{e^{i\theta}}{a}\right|
    \le \log 2.
\]
Consequently,
\[
    \frac{1}{2\pi}\int_0^{2\pi}|k_a(\theta)|\,d\theta
    =
    \frac{1}{\pi}\int_0^{2\pi}(k_a(\theta))_+\,d\theta
    \le 2\log 2.
\]

Finally,
\[
    \log|re^{i\theta}-\zeta|-\log^+|\zeta|
    =
    k_a(\theta)
    +
    \log r+\log^+|a|-\log^+|\zeta|,
\]
and
\[
    0
    \le
    \log r+\log^+|a|-\log^+|\zeta|
    \le
    \log r.
\]
The result follows.
\end{proof}

Define the finite measure
\begin{equation*}
    d\mu(\zeta)
    :=
    \frac{1}{2\pi}
    K(\zeta)\,d\operatorname{vol}(\zeta)
\end{equation*}
and the logarithmic potential
\begin{equation}\label{RieszPotentialDefinition}
    P(z)
    :=
    \int_{\mathbb C}
    \left(
        \log|z-\zeta|-\log^+|\zeta|
    \right)
    \,d\mu(\zeta).
\end{equation}
This integral is well defined for every $z\in\mathbb C$. Indeed, on
bounded sets the measure $\mu$ has a smooth density and logarithmic
singularities are locally integrable. For $|\zeta|$ sufficiently large
relative to $|z|$, the integrand is bounded and tends to zero as $|\zeta|\to\infty$. Since
\begin{equation*}
    \Delta_{\mathbb{C}} \log|z-\zeta|
    =
    2\pi\delta_\zeta,
\end{equation*}
we have
\begin{equation*}
    \Delta_{\mathbb{C}} P
    =
    \Delta_{\mathbb{C}} h
\end{equation*}
in the sense of distributions. Therefore
\begin{equation*}
    q:=h-P
\end{equation*}
is harmonic on all of $\mathbb C$.

\begin{proposition}
\label{HarmonicRemainderConstantProp}
The entire harmonic function $q$ is constant.
\end{proposition}

\begin{proof}
By Lemma~\ref{LogKernelCircularL1Lemma} and Tonelli's theorem,
\begin{equation*}
    \frac{1}{2\pi}
    \int_0^{2\pi}|P(re^{i\theta})|\,d\theta
    \le
    \int_{\mathbb C}
    \frac{1}{2\pi}
    \int_0^{2\pi}
    \left|
        \log|re^{i\theta}-\zeta|-\log^+|\zeta|
    \right|
    \,d\theta\,d\mu(\zeta)
    \le
    C\mu(\mathbb C)(1+\log r)
\end{equation*}
for every $r\ge 1$.

By~\eqref{hlogupperbound},
\begin{equation*}
    h(re^{i\theta})
    \le
    C(1+\log r).
\end{equation*}
Since $q=h-P$, we have
\begin{equation*}
    q^+
    \le
    h^++P^-,
\end{equation*}
and consequently
\begin{equation}\label{qPositiveCircularGrowth}
    \frac{1}{2\pi}
    \int_0^{2\pi}
    q^+(re^{i\theta})\,d\theta
    \le
    C(1+\log r).
\end{equation}
Because $q$ is harmonic, the mean-value property gives
\begin{equation*}
    \frac{1}{2\pi}
    \int_0^{2\pi}
    q(re^{i\theta})\,d\theta
    =
    q(0).
\end{equation*}
It follows from~\eqref{qPositiveCircularGrowth} that
\begin{equation}\label{qCircularL1Growth}
    \frac{1}{2\pi}
    \int_0^{2\pi}
    |q(re^{i\theta})|\,d\theta
    \le
    C(1+\log r).
\end{equation}

Since $\mathbb C$ is simply connected, $q$ admits a global harmonic
conjugate and is the real part of an entire holomorphic function.
Thus
\begin{equation*}
    q(re^{i\theta})
    =
    a_0+
    \sum_{n=1}^{\infty}
    r^n
    \bigl(
        a_n\cos(n\theta)+b_n\sin(n\theta)
    \bigr),
\end{equation*}
where the series converges uniformly on every compact disk. For
$n\ge 1$, orthogonality gives
\begin{equation*}
    a_nr^n
    =
    \frac{1}{\pi}
    \int_0^{2\pi}
    q(re^{i\theta})\cos(n\theta)\,d\theta, \qquad
    b_nr^n
    =
    \frac{1}{\pi}
    \int_0^{2\pi}
    q(re^{i\theta})\sin(n\theta)\,d\theta.
\end{equation*}
Therefore, by~\eqref{qCircularL1Growth},
\begin{equation*}
    |a_n|+|b_n|
    \le
    C\frac{1+\log r}{r^n}.
\end{equation*}
Letting $r\to\infty$, we obtain
\begin{equation*}
    a_n=b_n=0
\end{equation*}
for every $n\ge 1$. Hence $q\equiv a_0$ is constant.
\end{proof}

We have therefore proved that
\begin{equation}\label{hRieszRep}
    h(z)
    =
    \frac{1}{2\pi}
    \int_{\mathbb C}
    \left(
        \log|z-\zeta|-\log^+|\zeta|
    \right)
    K(\zeta)\,d\operatorname{vol}(\zeta),
\end{equation} up to an additive constant, which we may take to be zero.

\subsection{Proof of almost-monotonicity} We restate and prove Theorem~\ref{mainthm2DmanifoldsChapter1}.

\begin{theorem}\label{mainthm2Dmanifolds} Let $M$ be a manifold satisfying Assumption~\ref{manifoldAssumption} and let $\alpha = 1 - \frac{1}{2\pi}\int_{M} K \,d\operatorname{vol}$ be as in~\eqref{alphadef}. Let $u_1, u_2 \in H^1_{\text{loc}}(M)$ be continuous, nonnegative, and subharmonic functions on $M$, with vanishing product. Let $p \in M$ with $u_1(p) = u_2(p) = 0$, and let $G$ be the conformal Green's function with pole at $p$. Then, the function $\Phi$ \eqref{Phimflddef} satisfies\begin{align*}
    \int_{\tau}^{T} \frac{\Phi'(r)}{\Phi(r)} \,dr &\ge -C + 2\log \alpha,
\end{align*} uniformly in $0 < \tau < T$, for an absolute constant $C \in \mathbb{R}.$
\begin{proof} We begin our argument as in the proofs of Proposition~\ref{planaralmostmonPoleBoundary} and Proposition~\ref{planaralmostmonPoleInterior}. If we define the eigenvalues $\lambda_i(r)$ as the largest constant such that \begin{align}\label{lambdaeigOmegar}
    \int_{\partial \Omega_r} |\nabla_T u|^2 \frac{1}{|\nabla R|} \,d\sigma &\ge \lambda_i(r) \int_{\partial \Omega_r} u^2 |\nabla R| \,d\sigma
\end{align} for all $u \in H_0^1(\Gamma_i(r))$, where $\Gamma_i(r)$ is the support of $u_i$ in $\partial \Omega_r$, then\begin{align}\begin{split}\label{ACFwithtchoiceMfld}
    \frac{\Phi'(r)}{\Phi(r)} &\ge 2(\sqrt{\lambda_1(r)} + \sqrt{\lambda_2(r)}) - \frac{2|\Omega_r|'}{|\Omega_r|}.
\end{split}
\end{align} We now simplify the expressions that defined our eigenvalue $\lambda_i(r)$ in~\eqref{lambdaeigOmegar}. We have \begin{align*}
    \int_{\partial \Omega_r} |\nabla_T u|^2 \frac{1}{|\nabla R|} \,d\sigma &= \int_{0}^{2\pi} \left|\left\langle  du, \frac{1}{|\partial_{\theta} x|} \,d\theta \right\rangle\right|^2 \frac{1}{|\nabla R|} |\partial_{\theta} x| \,d\theta.
\end{align*} From~\eqref{GthetaCoords}, $|\partial_{\theta} x| = \frac{1}{|\nabla G|}$. Therefore, \begin{align*}
    \int_{\partial \Omega_r} |\nabla_T u|^2 \frac{1}{|\nabla R|} \,d\sigma &= \frac{1}{r}\int_{0}^{2\pi} (\partial_{\theta} u)^2 \,d\theta.
\end{align*} Similarly, \begin{equation*}
    \int_{\partial \Omega_r} u^2 |\nabla R| \,d\sigma = \int_{0}^{2\pi} u^2 |\nabla R| |\partial_{\theta} x| \,d\theta = r\int_{0}^{2\pi} u^2 \,d\theta.
\end{equation*} It follows that $\lambda_i(r)$ is the largest constant such that \begin{align*}
    \int_{0}^{2\pi} (\partial_{\theta} u)^2 \,d\theta &\ge r^2 \lambda_i(r) \int_{0}^{2\pi} u^2 \,d\theta
\end{align*} for all functions $u = u(\theta) \in H^1([0, 2\pi])$ supported in the support of $u_i(r, \cdot)$, where we use $(r, \theta)$ coordinates. We now have unweighted Wirtinger inequalities as in the Friedland-Hayman inequality. The optimal partition is when $u_1, u_2$ are sine functions supported on $[0, \pi], [\pi, 2\pi]$ respectively, and it follows that \begin{align*}
    \sqrt{\lambda_1(r)} + \sqrt{\lambda_2(r)} &\ge \frac{2}{r}.
\end{align*} From~\eqref{ACFwithtchoiceMfld}, we obtain \begin{equation}\label{mfldestimateafterACF}
    \frac{\Phi'(r)}{\Phi(r)} \ge \frac{4}{r} - \frac{2|\Omega_r|'}{|\Omega_r|}.
\end{equation}Let us write $|\Omega_r|$ in $(r, \theta)$ coordinates. From~\eqref{GthetaCoords}, we have \begin{align*}
    |\Omega_r| &= \int_{0}^{r}\int_{0}^{2\pi} t |\nabla R|^{-2}(t, \theta) \,d\theta \,dt.
\end{align*} Let $h(r) = \int_{0}^{2\pi}|\nabla R|^{-2}(r, \theta) \,d\theta,$ so that $|\Omega_r| = \int_{0}^{r} t h(t) \,dt.$ Recall~\eqref{hRieszRep}, \begin{align*}
    |\nabla R|(r, \theta) &= \exp\left(\int_{\mathbb{C}} (\log |z-\zeta| - \log^{+}|\zeta|) \,d\mu(\zeta) \right),
\end{align*} where $\,d\mu(\zeta) = \frac{1}{2\pi}K \,d\mu(\zeta)$. This may be compared to the Schwarz-Christoffel formula~\eqref{SchwarzChristoffelMeasure} and its role in the proof of Theorem~\ref{planaralmostmonPoleBoundaryChapter1}. Note that here we have a measure on $\mathbb{C}$ rather than $\mathbb{R}$. Arguing as in Proposition~\ref{planaralmostmonPoleBoundary}, we may reduce to the case when $\mu$ is a finite Borel measure assigning no mass to a neighborhood of the origin and to the case $\tau = 0^+$. By~\eqref{mfldestimateafterACF}, we have \begin{equation*}
    \int_{0^+}^{T}\frac{\Phi'(r)}{2\Phi(r)} \,dr \le \int_{0^+}^{T} \left(\frac{|\Omega_r|'}{|\Omega_r|} - \frac{2}{r} \right)\,dr = \log\left(\frac{|\Omega_T|}{T^2}\right) - \lim_{\tau \to 0} \log\left(\frac{|\Omega_{\tau}|}{\tau^2}\right).
\end{equation*} We have \begin{equation*}
    \lim_{\tau \to 0}\frac{|\Omega_{\tau}|}{\tau^2} = \frac{1}{2}h(0) = \pi \exp\left(-2\int_{\mathbb{C}}\left( \log |\zeta| - \log^+ |\zeta| \right) \,d\mu(\zeta) \right).
\end{equation*} Recall that $\mu$ assigns no mass to a neighborhood of the origin. Without loss of generality, we assume that $\mu\left({B_1}\right) = 0$, so that $\int_{\mathbb{C}}\left(\log |\zeta| - \log^{+}|\zeta|\right) \,d\mu = 0$ and \begin{align*}
    \lim_{\tau \to 0} \frac{|\Omega_{\tau}|}{\tau^2} &= \pi.
\end{align*} Therefore, it suffices to prove \begin{align*}
    \frac{|\Omega_T|}{T^2} &\lesssim \frac{1}{\alpha}.
\end{align*} As in the proof of Proposition~\ref{planaralmostmonPoleBoundary}, it suffices to prove \begin{align}\label{inthT2Tmfld}
    \int_{I} h(r) \,dr &\lesssim \frac{T}{\alpha}
\end{align} for all intervals $I = [T, 2T]$. Recalling that $\mu({B_1}) = 0$, we write \begin{align*}
    \int_{I} h(r) \,dr &= \int_{T}^{2T} \int_{0}^{2\pi} \exp\left(2\int_{\mathbb{C}} \log \left|\frac{\zeta}{re^{i\theta} - \zeta}\right| \,d\mu(\zeta) \right) \,d\theta \,dr.
\end{align*} Let us write $\zeta = t e^{i\xi}$ in polar coordinates, with $t > 0$ and $\xi \in [0, 2\pi).$ If $|\zeta| \notin [T/2, 4T]$, \begin{align}\label{ratiosizewithzeta}
    \left|\frac{\zeta}{re^{i\theta} - \zeta}\right| &\lesssim 1.
\end{align} Let $J = [T/2, 4T]$ and $\mu_{J}$ denote the restriction of $\mu$ to the annulus $\{\zeta \in \mathbb{C}: |\zeta| \in J\}$. Note that \begin{align*}
    \left|re^{i\theta} - \zeta\right|^2 &= r^2 - 2rt \cos(\theta - \xi) + t^2.
\end{align*} If $\cos(\theta - \xi)<0,$ \begin{align*}
    r^2 - 2rt \cos(\theta - \xi) + t^2 &\ge t^2.
\end{align*} Therefore, given a fixed $\theta$ and $\cos(\theta - \xi) < 0$, \begin{align*}
    \left|\frac{\zeta}{re^{i\theta} - \zeta}\right| &\le 1,
\end{align*} and no positive contribution to the $\zeta$-integral is made for such $\zeta$.  

When $\cos(\theta - \xi) \ge 0$, then $\theta - \xi \in \left[-\frac{\pi}{2}, \frac{\pi}{2}\right] + 2\pi \mathbb{Z}.$ Let $\tilde{\xi}$ be chosen so that $\xi - \tilde{\xi} \in 2\pi \mathbb{Z}$ and $\theta - \tilde{\xi} \in \left[-\frac{\pi}{2}, \frac{\pi}{2}\right]$. Then, either $\tilde{\xi} = \xi$ or $\tilde{n} = \xi + 2\pi$. On this range, we have the elementary inequality $\cos (\theta - \tilde{\xi}) \le 1 - \frac{1}{4}\left(\theta - \tilde{\xi}\right)^2.$ By~\eqref{ratiosizewithzeta}, we may now write \begin{align*}
    &\int_{T}^{2T} \int_{0}^{2\pi} \exp\left(2\int_{\mathbb{C}} \log \left|\frac{\zeta}{re^{i\theta} - \zeta}\right| \,d\mu(\zeta) \right) \,d\theta \,dr \\ &\lesssim \int_{T}^{2T} \int_{0}^{2\pi} \exp\left( \int_{\mathbb{C}} \chi_{\{\cos(\theta - \xi) \ge 0 \}}(\zeta) \log \frac{t^2}{(r-t)^2 + \frac{1}{2}rt (\theta - \tilde{\xi})^2} \,d\mu_J(\zeta)\right) \,d\theta \,dr \\ &\le \int_{T}^{2T} \int_{0}^{2\pi} \exp\left( \int_{\mathbb{C}} \chi_{\{\cos(\theta - \xi) \ge 0 \}}(\zeta) \log K_{\zeta}(r, \theta)\,d\mu_J(\zeta)\right) \,d\theta \,dr,
\end{align*} where \begin{equation*}K_{\zeta}(r, \theta) := \frac{16T^2}{(r-t)^2 + \frac{1}{4}T^2 (\theta - \tilde{\xi})^2}. \qquad (\zeta = te^{i\xi})\end{equation*} Set $B := \mu_J(\mathbb{C}) \le \mu(\mathbb{C}) = 1 - \alpha.$ If $B = 0$, the estimate~\eqref{inthT2Tmfld} follows immediately, so assume $B > 0$. By Jensen's inequality, \begin{equation*}
    \exp\left(\int_{\mathbb{C}} \chi_{\{\cos(\theta - \xi) \ge 0\}}(\zeta) \log K_{\zeta}(r, \theta) \,d\mu_J(\zeta) \right) \le \int_{\mathbb{C}} K_{\zeta}(\theta, r)^{B \chi_{\{\cos(\theta - \xi) \ge 0\}}(\zeta)} \,\frac{d\mu_J(\zeta)}{B} \lesssim \int_{\mathbb{C}} K_{\zeta}(\theta, r)^{B} \,\frac{d\mu_J(\zeta)}{B},
\end{equation*} where the last estimate follows from $K_{\zeta}(\theta, r) \gtrsim 1$ for $r \in [T, 2T], \theta \in [0, 2\pi]$. Since $\tilde{\xi} \in \{\xi, \xi + 2\pi\}$, we have \begin{align*}
    \int_{I} h(r) \,dr &\lesssim \int_{T}^{2T} \int_{0}^{2\pi} \int_{\mathbb{C}}\left(\frac{16T^2}{(r-t)^2 + \frac{1}{4}T^2 (\theta - {\xi})^2} \right)^B \,\frac{d\mu_J(\zeta)}{B} \,d\theta \,dr \\ &+ \int_{T}^{2T} \int_{0}^{2\pi} \int_{\mathbb{C}} \left(\frac{16T^2}{(r-t)^2 + \frac{1}{4}T^2 (\theta - (\xi+2\pi))^2}\right)^B  \,\frac{d\mu_J(\zeta)}{B} \,d\theta \,dr.
\end{align*} It suffices to estimate the first integral. By Tonelli's theorem, \begin{align*}
    \int_{T}^{2T} \int_{0}^{2\pi} \int_{\mathbb{C}} \left(\frac{16T^2}{(r-t)^2 + \frac{1}{4}T^2 (\theta - {\xi})^2} \right)^B \,\frac{d\mu_J(\zeta)}{B} \,d\theta \,dr &= \int_{\mathbb{C}} \int_{T}^{2T} \int_{0}^{2\pi}\left(\frac{16T^2}{(r-t)^2 + \frac{1}{4}T^2 (\theta - {\xi})^2} \right)^B\,d\theta \,dr \,\frac{d\mu_J(\zeta)}{B}.
\end{align*} It therefore suffices to prove the estimate \begin{align*}
    \int_{T}^{2T} \int_{0}^{2\pi}\left(\frac{16T^2}{(r-t)^2 + \frac{1}{4}T^2 (\theta - {\xi})^2} \right)^B\,d\theta \,dr &\lesssim \frac{T}{1-B}
\end{align*} uniformly in $|\zeta| \in [T/2, 4T]$. This follows as with the corresponding estimate~\eqref{doubleintKProp4} in Proposition~\ref{planaralmostmonPoleBoundary}.

\end{proof}
\end{theorem}

\subsection{Fiala-Huber inequality}

In this section, we make all the assumptions of Theorem~\ref{mainthm2Dmanifolds}. We show how a different proof of almost-monotonicity for $\Phi$ in this manifold setting can be obtained using an isoperimetric inequality of Fiala \cite{Fiala}, who proved that if an analytic curve $C$ on a 2-dimensional analytic surface $M$ encloses a simply-connected domain $D$ of area $A$, and $M$ has nonnegative Gaussian curvature $K$, then \begin{align*}
    L^2 &\ge 4\pi A \left(1 - \frac{1}{2\pi} \int_{D} K \,d\operatorname{vol}\right).
\end{align*} Fiala's inequality was reproved and extended by Huber \cite{HuberIso}, who used conformal methods more in the spirit of this work. 

Since $G$ has no critical points by Proposition~\ref{GthetaCoordsProp}, $\partial \Omega_{r}$ is analytic for all $r > 0$. On the other hand, $F: M \to \mathbb{C}, (u, v) \to u + iv$ is a diffeomorphism by Proposition~\ref{GthetaCoordsProp}. Therefore, since the ball is simply connected in $\mathbb{C}$, $\Omega_{r}$ is simply connected in $M$. Using the Fiala-Huber inequality, we will prove the following.
\begin{theorem}\label{almostmonotFiala} Under the assumptions of Theorem~\ref{mainthm2Dmanifolds},
\begin{align*}
    \int_{0^+}^{T} \frac{\Phi'(r)}{\Phi(r)} \,dr &\ge -C + 6 \log \alpha
\end{align*} for an absolute constant $C \in \mathbb{R}$.
\end{theorem}

The $6\log \alpha$ in Theorem~\ref{almostmonotFiala} is suboptimal by Theorem~\ref{mainthm2Dmanifolds}, but it is nevertheless interesting that a comparable monotonicity result can be proven via this alternative route. 

\begin{proof}[Proof of Theorem~\ref{almostmonotFiala}] Following the reduction in Proposition~\ref{planaralmostmonPoleBoundary}, we may assume that the Riesz measure $\mu = \frac{1}{2\pi}K \,d\operatorname{vol}(\zeta)$ assigns zero mass to a neighborhood of origin, and without loss of generality we may assume $\mu({B_1}) = 0$. As in the proof of Theorem~\ref{mainthm2Dmanifolds}, it follows that \begin{align*}
    \lim_{\tau \to 0} \frac{|\Omega_{\tau}|}{\tau^2} &= \pi,
\end{align*} and \begin{align*}
    \int_{0^+}^{T} \frac{\Phi'(r)}{\Phi(r)} \,dr &\le \log\left(\frac{|\Omega_T|}{\pi T^2}\right).
\end{align*} It therefore suffices to prove \begin{align}\label{areatoT2cubicalpha}
    \frac{|\Omega_T|}{T^2} &\lesssim \frac{1}{\alpha^3}.
\end{align} By the Fiala-Huber inequality, \begin{align*}
    |\Omega_T| &\le \frac{|\partial \Omega_T|^2}{4\pi \alpha}.
\end{align*} Therefore,~\eqref{areatoT2cubicalpha} follows from \begin{align}\label{perimeterboundFiala}
    |\partial \Omega_T| &\lesssim \frac{T}{\alpha}.
\end{align} This is the content of the following lemma, in view of the identity \begin{align*}
    |\partial \Omega_{T}| &= T \int_{0}^{2\pi} |\nabla R|^{-1}(T, \theta) \,d\theta,
\end{align*} which follows from the coarea formula. Here, $(T, \theta)$ are $(r, \theta)$ coordinates. \end{proof}

\begin{lemma}
Let $\mu$ be a measure with $\mu({B_1}) = 0$ and $\mu(\mathbb{C}) \le 1 - \alpha$. Then, \begin{align*}
    \int_{0}^{2\pi} \exp\left(\int_{\mathbb{C}} \log \left|\frac{\zeta}{Te^{i\theta} - \zeta} \right| \,d\mu(\zeta) \right) \,d\theta &\lesssim \frac{1}{\alpha}.
\end{align*} uniformly in $T \ge 0$.
\begin{proof}
Let $B := \mu(\mathbb{C}),$ which we may without loss of generality assume to be positive. By Jensen's inequality, \begin{align*}
    \exp\left(\int_{\mathbb{C}} \log \left|\frac{\zeta}{Te^{i\theta} - \zeta} \right| \,d\mu(\zeta) \right) &\le \int_{\mathbb{C}} \left|\frac{T}{\zeta} e^{i\theta} - 1\right|^{-B} \,\frac{d\mu(\zeta)}{B}.
\end{align*} By Tonelli's theorem, it suffices to prove the estimate \begin{align*}\label{FialaOneFactorBound}
    \int_{0}^{2\pi}\left| \frac{T}{\zeta}e^{i\theta} - 1\right|^{-B} \,d\theta &\lesssim \frac{1}{1 - B}
\end{align*} uniformly in $\zeta \in \mathbb{C}\setminus \{0\}$ and $T \ge 0$. Note that by shifting the $\theta$ variable, we may assume $\zeta$ is real and positive. Therefore, equivalently, we may prove \begin{align*}
    \int_{-\pi}^{\pi}\left|e^{i\theta} - s\right|^{-B} \,d\theta &\lesssim \frac{1}{1-B} s^{-B},
\end{align*} uniformly for $s > 0.$ If $s \le \frac{1}{2}$, the estimate follow from  $\left|e^{i\theta} - s\right| \ge \frac{1}{2}$. Therefore, we may consider $s \ge \frac{1}{2}.$ Note that $1 - \cos \theta \ge \frac{1}{4}\theta^2$ on $[-\pi, \pi]$. Therefore, we have \begin{equation*}
    \left|e^{i\theta} - s\right|^2 = (1 - s)^2 + 2s(1 - \cos \theta) \ge \frac{1}{4} \theta^2.
\end{equation*} Therefore, we may write \begin{equation*}
    \left|e^{i\theta} - s\right|^2 \ge (1 - s)^2 + \frac{2}{\pi^2}\theta^2 \ge \frac{2}{\pi^2} \theta^2
\end{equation*}It follows that \begin{equation*}
    \int_{-\pi}^{\pi} \left|e^{i\theta} - s\right|^{-B} \,d\theta \le \int_{-\pi}^{\pi} \left(\frac{\theta}{2}\right)^{-B} \,d\theta \lesssim \frac{1}{1 - B}  \lesssim \frac{1}{1 - B}s^{-B},
\end{equation*} since $s \ge \frac{1}{2}$. This completes the proof of the lemma, of~\eqref{perimeterboundFiala}, and of Theorem~\ref{almostmonotFiala}.
\end{proof}
\end{lemma}

\section{Applications to the regularity theory for two-phase free boundary problems}

\subsection{Quasisymmetry} We recall the following definition.

\begin{definition}
If $(X,d_X)$ and $(Y,d_Y)$ are metric spaces, a homeomorphism
$f:X\to Y$ is called $\eta$-quasisymmetric for an increasing homeomorphism $\eta:[0,\infty)\to[0,\infty)$ if
\begin{equation*}
    \frac{d_Y(f(x),f(z))}{d_Y(f(y),f(z))}
    \le
    \eta\left(\frac{d_X(x,z)}{d_X(y,z)}\right)
\end{equation*}
for every triple of distinct points $x,y,z\in X$.    
\end{definition}

 Given $z \in \mathbb{C}$, we consider the map \begin{align*}
    F^{-1} = Q: \mathbb{C} \to M, \qquad z = re^{i\theta} \to (r, \theta) \in M,
\end{align*} where we use $(R, \theta)$ coordinates, so that $R((r, \theta)) = r.$ 

\begin{theorem}\label{mfldquasisymmetry} There exists an increasing homeomorphism $\eta_{\alpha}: [0, \infty) \to [0,\infty)$ depending only on $\alpha$, so that $Q$ is $\eta_{\alpha}-$quasisymmetric.
\begin{proof}
Since $K \ge 0$, the Bishop-Gromov theorem implies that \begin{align*}
    t \to \frac{\operatorname{vol}(B_t(p))}{\pi t^2}
\end{align*} is decreasing. Since $\lim_{t \to 0}\frac{\operatorname{vol}(B_t(p))}{\pi t^2} = 1$ and, by a result of Shiohama \cite{Shiohama}, $\lim_{t\to \infty} \frac{\operatorname{vol}(B_t(p))}{\pi t^2} = \alpha,$ we have \begin{equation}\label{geodesicballvolumebounds}
    \alpha \pi t^2 \le \operatorname{vol}(B_t(p)) \le \pi t^2
\end{equation} for all $t \ge 0.$ Therefore, $M$ is Ahlfors 2-regular. We now show that $M$ is linearly locally connected or LLC. Recall that to be LLC is to be $\lambda-$LLC for some $\lambda \ge 1.$ For this, two conditions must be satisfied. 
\begin{enumerate}
    \item ($\lambda-\text{LLC}_1$) For each pair of distinct points $x, y \in B_t(a)$, there is a continuum (a compact, connected set) $E \subset B_{\lambda t}(a)$ such that $x, y \in E.$
    \item ($\lambda-\text{LLC}_2$) For each pair of distinct points $x, y \in M\setminus B_t(a),$ there is a continuum $E \subset M\setminus B_{t/\lambda}(a)$ such that $x, y \in E$.
\end{enumerate}

We prove that $M$ is $1-\text{LLC}_1$. Indeed, since $M$ is complete, for $x, y\in B_t(a)$, we can choose minimizing geodesics from $x$ to $a$ and from $a$ to $y$. Joining these two paths gives a continuous path contained in $B_t(a)$ joining $x$ and $y$.

For the second condition, by rescaling, it is equivalent to prove that there exists $\lambda = \lambda(\alpha)$ such that for every $a \in M, t > 0$, any two points of $M\setminus B_{\lambda t}(a)$ can be joined by a continuum contained in $M\setminus B_{t}(a).$ Suppose the claim is false. Then for any $j\to \infty,$ there exists $a_j \in M, t_j > 0$ and two points $x_j, y_j \in M\setminus B_{jt_j}(a_j)$ which cannot be joined inside $M\setminus B_{t_j}(a_j).$ Equivalently, $x_j, y_j$ lie in two different components of $M\setminus B_{t_j}(a_j)$, and each component reaches distance at least $jt_j$ from $a_j$. Now rescale the metric by $t_j$, that is \begin{align*}
    d_j := t_j^{-1} d_g,
\end{align*} where $d_g$ denotes the original metric on $M$. Then, $(M, d_j, a_j)$ is again a complete, pointed space with nonnegative Gaussian curvature. We use pointed notation, so that $a_j$ is the basepoint around which we take the limit. In the metric $d_j$, the ball $B_{t_j}(a_j) = B_{d_g}(a_j, t_j)$ becomes $B_{d_j}(a_j, 1).$ Then, $x_j, y_j \in M\setminus B_{d_j}(a_j, j)$ but cannot be joined outside $B_{d_j}(a_j, 1).$ The curvature condition is preserved in the relevant sense, so that each $(M, d_j)$ is a complete Alexandrov space with curvature bounded below by $0$. For any ball $B_{d_j}(a, t)$ in $(M, d_j, a_j)$, the scale invariant volume bound~\eqref{geodesicballvolumebounds} implies a finite covering number by balls of radius $\varepsilon > 0$ independent of $j$. By Theorem 8.1.10 in Burago, Burago, and Ivanov \cite{BBIMetricGeo}, the class $(M, d_j, a_j)$ is precompact, so that, up to a subsequence, we may assume that \begin{align*}
    (M, d_j, a_j) \to (X, d_{\infty}, a_{\infty})
\end{align*} in Gromov-Hausdorff sense. Moreover, by Proposition 10.7.1 in the same text, $(X, d_{\infty}, a_{\infty})$ also has nonnegative curvature. 

The limit space $X$ has at least two unbounded components outside a fixed compact set. Therefore, since $X$ is complete, it contains a line. By the Toponogov splitting theorem, \begin{equation*}
    X \simeq \mathbb{R} \times Y,
\end{equation*} where $Y$ is also a nonnegatively curved Alexandrov space. Since $X$ is 2-dimensional, $Y$ is one dimensional. Note that $Y$ must be compact, otherwise $X$ would have only one end. Hence, for large $R$, balls of radius $R$ in $X$ have at most linear growth. This contradicts the quadratic growth that is inherited from~\eqref{geodesicballvolumebounds}. Therefore, for some $\lambda(\alpha)>0$, $M$ is $\lambda(\alpha)-\text{LLC}_{2}.$

We now have that $M$ is a complete, Ahlfors 2-regular, LLC space, which by Lemma~\ref{MdiffeoR2}, is additionally diffeomorphic to the plane. Therefore, by a theorem of Wildrick \cite{Wildrick}, $M$ is quasisymmetrically equivalent to $\mathbb{C}$, meaning that there exists a bijective homeomorphism \begin{align*}
    h: \mathbb{C} \to M,
\end{align*} which is ${\eta}$-quasisymmetric, for some increasing homeomorphism ${\eta}$ depending on the LLC constants, which here depend only on $\alpha$. Since $F$ is conformal, the inverse map $Q: \mathbb{C} \to M$ is conformal. Therefore, the map \begin{align*}
    H := h^{-1} \circ Q : \mathbb{C} \to \mathbb{C}
\end{align*} has bounded distortion depending on ${\eta}$. Therefore, from Tukia and V\"ais\"al\"a \cite{TukiaVaisala}, $H$ is $\psi$-quasisymmetric, where $\psi$ depends on $\alpha$. Therefore, $Q$ is $\eta_{\alpha}$-quasisymmetric, with $\eta_{\alpha} := \eta \circ \psi$, concluding the proof.
\end{proof}
\end{theorem}

\subsection{Lipschitz bound}

In this section, we prove Theorem~\ref{mfldregularitythmChapter1}. We begin with the ACF linear growth bound.

\begin{lemma}\label{interiorACFgrowthM}
Let $M$ satisfy Assumption~\ref{manifoldAssumption}. There exist constants $L_*=L_*(\alpha)>1$ and $C_*=C_*(\alpha)<\infty$ with the following property: Let $U\subset M$ be open, and let $u$ be a local minimizer of
\begin{equation}
    J[v]
    =
    \int_U \left(|\nabla v|^2 + 1_{\{v > 0\}}\right) \,d\operatorname{vol}.
\end{equation}
Let $z\in U$, and let $p\in\{u=0\}$ satisfy $d
    :=
    d_g(z,p)
    =
    \operatorname{dist}_g(z,\{u=0\}).$
Assume that $B_{2L_* d}(p)\subset U.$
Then
\begin{align}
    |u(z)|
    &\leq
    C_*d
    +
    C_*
    \left[
    \left(
    \int_{ B_{L_* d}(p)}|\nabla u^+|^2\,\,d\operatorname{vol}
    \right)
    \left(
    \int_{B_{L_* d}(p))}|\nabla u^-|^2\,\,d\operatorname{vol}
    \right)
    \right]^{1/4}.
\end{align}
\end{lemma}

\begin{proof}
We follow the interior argument in "Another Proof of Theorem~5.3" in Alt, Caffarelli, and Friedman \cite{ACF84}. The only point requiring comment is the uniformity of the constants on $M$. The Euclidean local estimates used in the ACF proof are replaced by their intrinsic analogues on geodesic balls. Indeed, $K\geq0$ gives $\operatorname{Ric}\geq0$, so, the scale-invariant local Poincar\'e inequality, Harnack inequality, and interior harmonic gradient estimates hold with dimensional constants. The Harnack inequality follows from Cheng-Yau, while the other two follow from work of Li and Schoen \cite{LiSchoen}. Caccioppoli estimates follow from the weak harmonic equation by the usual cutoff argument. Finally, the Ahlfors bounds~\eqref{geodesicballvolumebounds} replace the Euclidean volume estimates and make the density, nondegeneracy, and harmonic-replacement estimates quantitative with constants depending only on $\alpha$.

Thus the interior ACF growth argument applies on $M$, with Euclidean balls replaced by geodesic balls and with constants depending only on $\alpha$.
\end{proof}

\begin{proof}[Proof of Theorem~\ref{mfldregularitythmChapter1}]
Fix $L>1$. Choose $\Lambda=\Lambda(L,\alpha)>1$ sufficiently
large that $\eta_\alpha(\Lambda^{-1}) < \frac{1}{L},$
and set $\tilde{L} := \max\{L,\eta_\alpha(\Lambda)\}.$
We claim that if $x\in M$, $s=d(x,p)$, and $r=R(x)$, then
\begin{align}\label{manifoldGreenSandwichRegularity}
    B_{Ls}(p)
    \subset
   \Omega_{\Lambda r}
    \subset
    B_{\tilde{L}s}(p).
\end{align}
Indeed, let $y\in B_{Ls}(p)$ and suppose that
$R_p(y)\geq\Lambda r$. Quasisymmetry, applied with base point
$0$, gives
\begin{equation*}
    \frac{s}{d(y,p)}
    =
    \frac{d(Q(F(x)),Q(0))}
         {d(Q(F(y)),Q(0))}
    \leq
    \eta_\alpha
    \left(
        \frac{R(x)}{R(y)}
    \right)
    \leq
    \eta_\alpha(\Lambda^{-1})
    <
    \frac{1}{L},
\end{equation*}
contradicting $d(y,p)<Ls$. This proves the first inclusion.
If instead $y\in \Omega_{\Lambda r}$, then
\begin{equation*}
    \frac{d(y,p)}{s}
    \leq
    \eta_\alpha
    \left(
        \frac{R(y)}{R(x)}
    \right)
    <
    \eta_\alpha(\Lambda)
    \leq
    \tilde{L},
\end{equation*}
which proves the second inclusion.

By a covering argument and the natural rescaling of the
functional, it is enough to prove that for some
$\varrho=\varrho(\alpha)>0$,
\begin{align}\label{normalizedManifoldLipschitzEstimate}
    \sup_{B_\varrho(x_0)}
    |\nabla u|
    \leq
    C(\alpha)E^{1/2},
\end{align} where \begin{align*}
    E
    &:=
    \int_{B_1(x_0)}
    |\nabla u|^2\,d\operatorname{vol}.
\end{align*}

Choose $L = L_*(\alpha),$ where $L_*(\alpha)$ is as in Lemma~\ref{interiorACFgrowthM},
and let $\Lambda$ and $\tilde{L}$ be as in
\eqref{manifoldGreenSandwichRegularity}. Fix $s_0 :=({4\tilde{L}})^{-1}.$
We choose $\varrho=\varrho(\alpha)>0$ sufficiently small that,
whenever $z\in B_\varrho(x_0)$, $p\in B_{2\varrho}(x_0)$, and
$d_g(z,p)<\varrho$, one has
\begin{equation}\label{rhoChoicesManifoldRegularity}
    B_{\tilde{L}s_0}(p)
    \subset
    B_1(x_0), \qquad 
    \tilde{L}d(z,p)
    \leq
    Ls_0,
    \qquad
    B_{2Ld(z,p)}(p)
    \subset
    B_1(x_0).
\end{equation}

Let $z\in B_\varrho(x_0)$. If $d := \operatorname{dist}(z,\{u=0\}) \geq \varrho,$ then $u$ has a fixed sign and is harmonic in
$B_{\varrho/2}(z)$. Since $K\geq0$, the Bochner formula shows
that $|\nabla u|^2$ is subharmonic there. The mean-value
inequality of Li and Schoen \cite{LiSchoen} and the uniform lower volume bound~\eqref{geodesicballvolumebounds} therefore give
\begin{equation*}
    |\nabla u(z)|^2
    \leq
    \frac{C}
    {\operatorname{vol}(B_{\varrho/4}(z))}
    \int_{B_{\varrho/2}(z)}
    |\nabla u|^2\,d\operatorname{vol}
    \leq
    C(\alpha)E.
\end{equation*}
Thus it remains to consider the case $d < \varrho.$ Choose $p\in\{u=0\}$ with $d(z,p)=d$. Then
$p\in B_{2\varrho}(x_0)$. Set $r_z := R(z)$. Applying \eqref{manifoldGreenSandwichRegularity} at the scale
$d$ gives
\begin{align}\label{smallScaleManifoldSandwich}
    B_{Ld}(p)
    \subset
    \Omega_{\Lambda r_z}
    \subset
    B_{\tilde{L}d}(p).
\end{align}
Since $M$ is complete and noncompact, we may choose
$y_0\in M$ satisfying $d(y_0,p) = s_0.$ Writing $r_0:=R(y_0)$, a second application of
\eqref{manifoldGreenSandwichRegularity} gives
\begin{align}\label{outerScaleManifoldSandwich}
    B_{Ls_0}(p)
    \subset
    \Omega_{\Lambda r_0}
    \subset
    B_{\tilde{L}s_0}(p)
    \subset
    B_1(x_0).
\end{align}
By \eqref{rhoChoicesManifoldRegularity},
\begin{align*}
    \Omega_{\Lambda r_z}
    \subset
    B_{\tilde{L}d}(p)
    \subset
    B_{Ls_0}(p)
    \subset
    \Omega_{\Lambda r_0}.
\end{align*}
In particular, $r_z\leq r_0$. By the almost-monotonicity formula, Theorem~\ref{mainthm2DmanifoldsChapter1}, for the functional $\Phi$, we obtain
\begin{align}\label{PhiSmallOuterManifold}
    \Phi(\Lambda r_z)
    \leq
    C(\alpha)\Phi(\Lambda r_0).
\end{align} Recall, \begin{align*}
    \Phi(r) := \frac{1}{|\Omega_r|^2} 
    \int_{\Omega_r}|\nabla u^+|^2\,\,d\operatorname{vol}
    \int_{\Omega_r}|\nabla u^-|^2\,\,d\operatorname{vol}
\end{align*}

By \eqref{outerScaleManifoldSandwich} and~\eqref{geodesicballvolumebounds}, we have
\begin{align*}
    \operatorname{vol}(\Omega_{\Lambda r_0})
    &\geq
    \operatorname{vol}(B_{Ls_0}(p))
    \geq
    \pi\alpha L^2s_0^2.
\end{align*}
Moreover,
\begin{align*}
    \int_{\Omega_{\Lambda r_0}}
    |\nabla u^\pm|^2\,d\operatorname{vol}
    \leq
    E.
\end{align*}
It follows that
\begin{align*}
    \Phi(\Lambda r_0)
    \leq
    C(\alpha)E^2,
\end{align*}
and hence, by \eqref{PhiSmallOuterManifold},
\begin{align}\label{PhiSmallBoundManifold}
    \Phi(\Lambda r_z)
    \leq
    C(\alpha)E^2.
\end{align}

On the other hand, the second inclusion in
\eqref{smallScaleManifoldSandwich} and
\eqref{geodesicballvolumebounds} give
\begin{align*}
    \operatorname{vol}(\Omega_{\Lambda r_z})
    &\leq
    \operatorname{vol}(B_{\tilde{L}d}(p))
    \leq
    \pi \tilde{L}^2d^2.
\end{align*}
Combining this with \eqref{PhiSmallBoundManifold}, we obtain
\begin{equation*}
    \left(
        \int_{\Omega_{\Lambda r_z}}
        |\nabla u^+|^2\,d\operatorname{vol}
    \right)
    \left(
        \int_{\Omega_{\Lambda r_z}}
        |\nabla u^-|^2\,d\operatorname{vol}
    \right) \leq C(\alpha)E^2d^4.
\end{equation*}
The first inclusion in
\eqref{smallScaleManifoldSandwich} then yields
\begin{equation}\label{ballEnergyProductManifold}
    \left(
        \int_{B_{Ld}(p)}
        |\nabla u^+|^2\,d\operatorname{vol}
    \right)
    \left(
        \int_{B_{Ld}(p)}
        |\nabla u^-|^2\,d\operatorname{vol}
    \right)
    \leq
    C(\alpha)E^2d^4.
\end{equation}

The last condition in \eqref{rhoChoicesManifoldRegularity}
allows us to apply Lemma~\ref{interiorACFgrowthM}. Together
with \eqref{ballEnergyProductManifold}, it gives
\begin{equation}
\label{linearGrowthManifold}
    |u(z)| \leq
    C(\alpha)d
    +
    C(\alpha)
    \left[
    \left(
        \int_{B_{L_*d}(p)}
        |\nabla u^+|^2\,d\operatorname{vol}
    \right)
    \left(
        \int_{B_{L_*d}(p)}
        |\nabla u^-|^2\,d\operatorname{vol}
    \right)
    \right]^{1/4} \leq
    C(\alpha)d
    +
    C(\alpha)E^{1/2}d.
\end{equation} Now, since $p$ remains a fixed positive distance away from $\partial B_1(x_0)$, we may choose $s_1 = s_1(\alpha) > 0$ such that $B_{2s_1}(p) \subset B_1(x_0)$. A standard argument combining nondegeneracy, zero-phase density estimates with local boundedness, and the Poincaré inequality, yields
\[
\int_{B_{s_1}(p)}|\nabla u^+|^2\,d\operatorname{vol}
\ge c(\alpha)s_1^2.
\]
Consequently, $E\ge c(\alpha)>0$, so that, by~\eqref{linearGrowthManifold}, \begin{align*}
    |u(z)| &\le C(\alpha) E^{\frac{1}{2}} d.
\end{align*}

Finally, $u$ has a fixed sign and is harmonic in $B_d(z)$.
Applying the interior gradient estimate to the positive
harmonic function $u$ or $-u$, as appropriate, gives
\begin{align*}
    |\nabla u(z)|
    &\leq
    \frac{C}{d}|u(z)|
    \leq
    C(\alpha)E^{1/2}.
\end{align*}
Together with the estimate in the case
$\operatorname{dist}_g(z,\{u=0\})\geq\varrho$, this proves
\eqref{normalizedManifoldLipschitzEstimate} and the theorem.
\end{proof}

\begin{remark}
The same argument used to prove the Lipschitz estimate in Theorem~\ref{mfldregularitythmChapter1} can be used to recover the aforementioned theorem of Gemmer, Moon, and Raynor \cite{GMR}, who proved that the Lipschitz estimate holds up to the Neumann boundary of a bounded planar convex body. Theorem~\ref{planaralmostmonPoleBoundaryChapter1} replaces Theorem~\ref{mainthm2DmanifoldsChapter1}. For the bounded convex domain case, one applies Theorem~\ref{planaralmostmonPoleBoundaryChapter1} by working locally near the boundary and extending the convex domain far away so that it contains a cone of opening angle depending on only some geometric parameters of the convex set, such as inradius and outradius. The replacement of the quasisymmetry of $Q$, Theorem~\ref{mfldquasisymmetry}, is accomplished as follows. First, one observes that the conformal map $\varphi: \mathbb{H} \to \Omega$ extends to a quasiconformal map $\widehat{\varphi}: \mathbb{C} \to \mathbb{C}$ with distortion bounded in terms of $\alpha$, by the Ahlfors-Beurling theorem \cite{BAhlfors} or an elementary reflection argument. That $\widehat{\varphi}$ is $\eta_{\alpha}-$quasisymmetric for some increasing homeomorphism $\eta_{\alpha}: [0, \infty) \to [0, \infty)$ then follows by a theorem of Tukia and V\"ais\"al\"a \cite{TukiaVaisala}. Finally, to accomodate Neumann boundary conditions, the linear growth estimate Lemma~\ref{interiorACFgrowthM} is replaced by the corresponding estimate in Theorem 7.1 and Proposition 7.2 of Beck, Jerison, and Raynor \cite{BJR}.
\end{remark}

\section{Acknowledgements}

This work constitutes the main results of the author's PhD thesis, completed under David Jerison at MIT. The author expresses his sincere gratitude to David Jerison for his insight, encouragement, and mathematical vision pertaining to this work. The author acknowledges financial support from Simons Foundation Collaboration Grant 601948 DJ. The main results presented here are generated by the author. ChatGPT Pro was used to assist with the proof of Theorem~\ref{mfldquasisymmetry}. Otherwise, it was used to simplify proofs, check technical steps, aid in exposition, and find references.

\bibliographystyle{alpha}
\bibliography{biblio.bib}
\end{document}